\documentclass[12pt]{article}

\usepackage{geometry}
\usepackage{amsmath}
\usepackage{amsthm}
\usepackage{amsfonts,dsfont}
\usepackage{amssymb}
\usepackage{stmaryrd}
\usepackage{fancyhdr}
\usepackage{verbatim}
\usepackage{color}
\usepackage[toc,page]{appendix}

\theoremstyle{definition}\newtheorem{definition}{Definition}[section]
\theoremstyle{definition}\newtheorem{proposition}[definition]{Proposition}
\theoremstyle{definition}\newtheorem{remark}[definition]{Remark}
\theoremstyle{definition}\newtheorem{theorem}[definition]{Theorem}
\theoremstyle{definition}
\theoremstyle{definition}
\theoremstyle{definition}

\numberwithin{equation}{section}

\newcommand{\lrank}{\operatorname{rank}}
\newcommand{\img}{\operatorname{img}}
\newcommand{\Diff}{\operatorname{Diff}}

\newcommand{\R}{\mathbb{R}}
\newcommand{\E}{\mathcal{E}}

\newcommand{\pp}[2]{\frac{\partial #1}{\partial #2}}

\newcommand{\n}{^{(n)}}
\newcommand{\ii}{^{(\infty )}}

\newcommand{\SL}{\mathrm{SL}}

\newcommand{\hD}{\widehat{D}}
\newcommand{\hQ}{\widehat{Q}}
\newcommand{\hP}{\widehat{P}}

\newcommand{\tE}{\tilde{\mathcal{E}}}

\newcommand{\tQ}{\tilde{Q}}
\newcommand{\ckQ}{\check{Q}}
\newcommand{\ckmu}{\check{\mu}}

\newcommand{\tomega}{\tilde{\omega}}
\newcommand{\ckomega}{\check{\omega}}
\newcommand{\tmu}{\tilde{\mu}}
\newcommand{\tnu}{\tilde{\nu}}

\newcommand{\vv}{\mathbf{v}}

\newcommand{\bl}{\pmb{\lambda}}
\newcommand{\bt}{\pmb{\tau}}
\newcommand{\btn}{\bt\n}

\newcommand{\bs}{\pmb{\sigma}}
\newcommand{\tbs}{\tilde{\bs}}
\newcommand{\tbt}{\tilde{\bt}}
\newcommand{\tvarrho}{\tilde{\varrho}}
\newcommand{\bsn}{\bs\n}

\newcommand{\op}{\overline{p}}
\newcommand{\ou}{\overline{u}}
\newcommand{\ox}{\overline{x}}

\newcommand{\D}{\mathcal{D}}

\newcommand{\G}{\mathcal{G}}

\newcommand{\J}{\mathcal{J}}
\newcommand{\rJ}{\mathrm{J}}

\newcommand{\Ji}{\J^\infty}

\newcommand{\ji}{{\rm j}_\infty}
\newcommand{\bO}{\mathbf{\Omega}}

\newcommand{\bvtheta}{\boldsymbol{\vartheta}}

\def\comp{\raise 1pt \hbox{$\,\scriptstyle\circ\,$}}

\def\dim{\mathop{\rm dim}\nolimits}

\allowdisplaybreaks[4]

\begin{document}

%%%%%%%%%
% Opening page
%%%%%%%%% 

\thispagestyle{fancy}
\fancyhead{}
\fancyfoot{}
\renewcommand{\headrulewidth}{0pt}
\cfoot{\thepage}
\rfoot{\today}

\vskip 1cm
\begin{center} {\Large Point Equivalence of Second-Order ODEs: Maximal
    Invariant Classification Order}
%\vskip 1cm
%Francis Valiquette
\vskip 1cm

\begin{tabular*}{1.0\textwidth}{@{\extracolsep{\fill}} ll}
Robert Milson\footnotemark[1] & Francis Valiquette\footnotemark[2]  \\
Department of Mathematics and Statistics & Department of Mathematics\\
Dalhousie University & SUNY at New Paltz \\
Halifax, Nova Scotia, Canada \quad B3H 3J5 & New Paltz, NY \quad 12561\\
{\tt milson@mathstat.dal.ca} & {\tt valiquef@newpaltz.edu} \\
{\tt http://www.mathstat.dal.ca/$\sim$milson/} &{\tt http://www2.newpaltz.edu/$\sim$valiquef}
\end{tabular*}
\end{center}

\footnotetext[1]{{\it Supported by NSERC grant
RGPIN-228057-2009.}}
\footnotetext[2]{{\it Supported in part by an AARMS Postdoctoral Fellowship.}}

  \vskip 0.5cm\noindent {\bf
  Mathematics subject classification (MSC2010)}: 53A55

\vskip 1cm

\abstract{We show that the local equivalence problem of second-order
  ordinary differential equations under point transformations is
  completely characterized by differential invariants of order at most
  10 and that this upper bound is sharp.  We also demonstrate that, modulo
  Cartan duality and point transformations, the Painlev\'e--I equation
  can be characterized as the simplest second-order ordinary differential equation belonging to the
  class of equations requiring 10th order jets for their
  classification.  }

%%%%%
\section{Introduction}
%%%%%

This paper is concerned with the local equivalence of second-order ordinary differential equations (ODEs) under point transformations.  This is a classical problem that has been extensively studied.  This is particularly true of the linearization problem which consists of determining when an equation $u_{xx} = f(x,u, u_x)$ is locally equivalent to $u_{xx}=0$.  Sophus Lie was the first to observe that the equation had to be cubic in the first order derivative
$$
u_{xx} = K(x,u)\, u_x^3 + L(x,u)\, u_x^2 + M(x,u)\, u_x + N(x,u)
$$
to be linearizable, \cite{L-1924}.  Precise conditions on the coefficients  $K$, $L$, $M$, $N$ were later determined by Liouville, \cite{L-1889}.  Subsequently, equivalent linearization conditions were found by many authors, \cite{C-1955,G-1989,GTW-1989,I-2004,K-1989,Q-2007,S-1997,S-1998,T-1896,Y-2004}.  Tresse was the first to give a complete generating set of differential invariants for generic second-order ODEs not constraint by differential relations, \cite{T-1896}.  He also fully characterized the equations admitting a point symmetry group (For a modern geometrical account of Tresse's paper we refer the reader to \cite{K-2009}.).  Another facet of the problem that has attracted considerable attention is the local classification of the Painlev\'e transcendents, \cite{KLS-1985,K-1989,K-2013}.  To the best of our knowledge, none of the aforementioned references studies all the different branches of the equivalence problem.  This is understandable as this is a computationally demanding task to do manually.  But with computer algebra systems becoming more and more efficient, a wide range of equivalence problems can now be codified.  It then becomes important to establish an upper bound on the number of iterations the algorithm has to go through to guarantee a complete solution.   %The complete solution requires the computation of a generating set of differential invariants for each branches, some requiring higher order invariants than others.  
In geometry, and particularly in general relativity, \cite{K-1980,MP-2008}, it is common to search for the highest order differential invariants encountered in the solution of an equivalence problem.  In this paper we do the same for the point equivalence problem of second-order ODEs.  To determine the highest order differential invariants occurring in the solution of the equivalence problem we survey the different branches of the equivalence problem and focus our attention on the most singular ones as the highest order invariants will occur in these branches.
%These invariants are found by analyzing the most singular branches of the equivalence problem.

There are several ways of finding these highest order invariants.  Based on one's preference, it is possible to use Lie's infinitesimal method, \cite{O-1993}, the theory of $G$-structures, \cite{K-1989,O-1995}, or the method of equivariant moving frames, \cite{OP-2008,V-2012}.  We decided to employ the theory of equivariant moving frames to take advantage of the symbolic and algorithmic nature of the method.  The solution relies on the \emph{universal recurrence relations} which symbolically determine the exterior derivative of differential invariants.  Very little information is needed to write down these equations.  It only requires the knowledge of the infinitesimal generator of the equivalence pseudo-group and the choice of a cross-section to the pseudo-group orbits.  In particular, the coordinate expressions for the invariants are not required.  Also, the computations only involve differentiation and linear algebra and these are well handled by symbolic softwares.  In our case, we used {\sc Mathematica} to implement the computations.

The main result of this paper establishes that any regular second order ODE is classified, relative to point transformations, by its 10th order jets.  Furthermore, this bound is sharp, meaning that there exist regular ODEs that are not classified by 9th order jets.  Furthermore, we show that every equation having maximal invariant classification (IC) order of 10 is equivalent, modulo point transformations and Cartan duality, to an equation of the form
\begin{equation}\label{eq:max_realizing_ode}
u_{xx} = 6u^2 + g(x) \quad\text{with}\quad D_{x}^2[g(x)^{-1/4}] \neq 0,
\end{equation}
the ``simplest" of which is the Painlev\'e--I equation $u_{xx}=6u^2+x$ whose 10th order classifying invariant vanishes.  The branch in which equation \eqref{eq:max_realizing_ode} occurs can also be found in the works of Kamran, \cite{K-1989}, Morozov, \cite{M-2010}, and Sharipov, \cite{S-1998}, though none of them have studied the question of maximal invariant classification order.  

%%%%%
\section{Formalization of the problem and results}\label{section:main}
%%%%%

Following standard practices, \cite{K-1989,O-1995}, we let
$$
p=u_x,\qquad q=u_{xx}
$$
denote the first and second order derivatives of a single variable function $u=u(x)$.   Then, two second-order ordinary differential equations
\begin{equation*}
%\label{2nd order ODEs}
q=f(x,u,p)\quad\text{and}\quad Q=F(X,U,P),\qquad p=u_x,\, q=u_{xx},\; P=U_X,\, Q=U_{XX},
\end{equation*}
are (locally) point equivalent if there exists a local diffeomorphism of the plane
\begin{equation}
  \label{eq:ptxform}
  (X,U)=\psi(x,u),\quad \psi\in \Diff(\R^2),
\end{equation}
such that
\begin{equation}
  \label{eq:Qtarget}
  F(X,U,P) = \hQ(p,f(x,u,p),X_x, X_u, U_x, U_u , X_{xx} , X_{xu}, X_{uu},
  U_{xx},U_{xu}, U_{uu})
\end{equation}
where
\begin{align}
\label{eq:PQdef}
&\begin{aligned}
  P&= \hP(p, X_x, X_u, U_x, U_u) = \frac{\hD_x U}{\hD_x X} = \frac{p\,
    U_u + U_x}{p\,X_u + X_x},\\
  Q &= \hQ(p,q,X_x, X_u, U_x, U_u , X_{xx} , X_{xu}, X_{uu},
  U_{xx},U_{xu}, U_{uu})\\
\end{aligned}\\ \nonumber
&=\frac{ \hD_xP}{ \hD_x X} = \frac{\hD_x^2U \cdot \hD_x X - \hD_x U
  \cdot \hD_x^2 X}{(\hD_x X)^3} \\ \nonumber
&= \frac{(p\, X_u + X_x)\left( p^2\,
    U_{uu} + 2 p\, U_{ux} + U_{xx}\right)- (p\, U_u + U_x)\left(p^2\, X_{uu}
    + 2p\, X_{ux} + X_{xx}\right)}{(p\, X_u + X_x)^{3}}\\ \nonumber
&\quad\quad + \frac{U_u X_x - U_x X_u}{(p\, X_u + X_x)^{3}}\, q,
\end{align}
describe the transformation law for $p=u_x$ and $q=u_{xx}$, and
\begin{equation*}
 \hD_x = \partial_x + p\,\partial_u + q\,\partial_p
\end{equation*}
is the truncation of the usual total derivative operator.  The equations \eqref{eq:PQdef} together with the usual contact
conditions, constitute a quasi-linear system of PDEs in the pseudo-group jet
variables $X,U, X_x, X_u, U_x, U_u$.  This system is over-determined,
owing to higher order integrability conditions that, for sufficiently
high order, govern the outcome of the equivalence problem.  These
integrability conditions take the form of equalities between differential invariants
of the two equations leading to Definition \ref{def:ICorder} below.

Taking repeated derivatives of \eqref{eq:Qtarget} with respect to
$X,U,P$ yields the following necessary conditions for equivalence:
\begin{gather*}
%  \label{eq:Qijkrel}
  F_{ijk}(X,U,P)=\frac{\partial^{i+j+k} F(X,U,P)}{\partial X^i\partial U^j\partial
    P^k} = \hQ_{ijk}(p,q,q_x, q_u, q_p, \ldots; X_x, X_u,
  U_x, U_u,\ldots),
 % \intertext{where}
 %  \label{eq:Qijkdef}
 %  \tQ_{ijk}(x,u,p , X_x, X_u ,U_x, U_u,\ldots ) =
 %  \hQ_{ijk}(p,q\n(x,u,p),\psi^{(n+2)}).
\end{gather*}
where
\begin{equation*}
%  \label{eq:hQdef}
  \hQ_{ijk}(p,q^{(i+j+k)};\psi_0^{(i+j+k+2)}),\qquad \psi
  \in \Diff(\R^2),
\end{equation*}
describes the transformation law for the partial derivatives 
$$
q_{ijk}=\frac{\partial^{i+j+k}q}{\partial x^i \partial u^j \partial p^k}
$$ 
under the point transformations \eqref{eq:ptxform}, and
where
\begin{equation}
  \label{eq:qndef}
  \begin{aligned}
    q\n &= \{ q_{ijk} \colon 0\leq i+j+k \leq n \},\qquad q_{ijk} \colon\hskip -0.2cm=
    q_{x^iu^jp^k};\\
    \psi\n &= \{ (X_{ij}, U_{ij})\colon 0\leq i+j \leq n \},\quad\,
    X_{ij} \colon\hskip -0.2cm=  X_{x^iu^j},\;\; U_{ij} \colon\hskip -0.2cm= U_{x^iu^j};\\
    \psi_0\n &= \{ ( X_{ij}, U_{ij}) \colon 1\leq i+j \leq n\} ;
\end{aligned}
\end{equation}
denote the indicated jets.

\begin{definition}\label{def:ICorder}
  We say that a second-order ODE $q=f(x,u,p)$ is classified by $n$th
  order jets if the \emph{algebraic} consistency of the system
  \begin{align*}
    P &= \hP(p, \psi_0^{(1)}) ,\\
    F_{ijk}(X,U,P) &= \hQ_{ijk}(p,f^{(i+j+k)}(x,u,p),
    \psi_0^{(i+j+k+2)}) ,\quad 0\leq i+j+k \leq n;
  \end{align*}
  is sufficient for the existence of a point transformation relating
  $q=f(x,u,p)$ and $Q=F(X,U,P)$.  We call the smallest such $n$ the
  \emph{IC (invariant classification) order} of the differential
  equation.
\end{definition}
\noindent
\textbf{Note:} In the formulation of the above definition it must be
understood that the pseudo-group variables $(X_x,X_u,U_x,U_u,\ldots)$
are to be treated as auxiliary independent variables rather than
functions of $x$ and $u$.

%The equivalence problem for second-order ordinary differential
%equations under point transformations was first considered by Lie,
%Tresse and Cartan in \cite{C-1955,L-1924,T-1896}, and was revisited by
%many subsequent authors; see \cite[Chapter 12]{O-1995} for a modern
%treatment and \cite{NS-2003} for a recent application.  

The question that motivates us here is the following:
\begin{quote}
  \emph{What is the maximal jet order required for the invariant
    classification of a second-order ordinary differential equation up
    to local point transformations?}
\end{quote}

By way of an example, the well-known Linearization Theorem for
second-order ordinary differential equations, \cite{G-1989,
  GTW-1989,O-1995, SML-1987}, states that $q=f(x,u,p)$ is point
equivalent to the trivial equation $Q=U_{XX}=0$ if and only if the
fourth-order relative\footnote{A relative invariant is a function whose value is multiplied by a certain factor, known as a multiplier, under pseudo-group transformations.  An invariant is a relative invariant with multiplier equal to one.} invariants
\begin{equation}\label{linearization conditions}
\begin{gathered}
q_{pppp}\equiv 0,\\
\hD_x^2 (q_{pp}) - 4 \hD_x (q_{u p}) - q_{p} \hD_x
(q_{pp}) + 6 q_{uu} - 3q_u q_{pp} + 4q_{p} q_{u p}\equiv 0,
\end{gathered}
\end{equation}
are identically zero.  From this it follows that the linearizable
class has IC order equal to $4$.  On the other hand, the invariant
classification of general second-order equations will require higher
order jets.  The complete answer regarding the maximal order is given
below in Theorem \ref{thm:main}.

A well-posed equivalence problem requires some notion of
regularity. Therefore, before proceeding further, we must impose a
technical rank assumption.  Owing to the covariance of the
transformation laws, \cite{OP-2008}, the functions $\hQ_{ijk}$ are invariant with
respect to point transformations.  For a smoothly defined ODE
$q=f(x,u,p)$, let 
\[q\n = f\n(x,u,p)\] denote the $n$th order jet of the defining
function.  The composition of $\hQ_{ijk}$ and $f\n$ produces an
invariant of the ODE,
\begin{equation}
  \label{eq:Qijkdef}
  Q_{ijk}=\hQ_{ijk}\left(p,f^{(i+j+k)}(x,u,p),\psi_0^{(i+j+k+2)}\right),
\end{equation}
which we call a \emph{lifted invariant}\footnote{By contrast,
  \emph{absolute differential invariants} are functions of $x,u,p$
  only.  To construct absolute invariants one eliminates, by normalization, the pseudo-group
  variables from the lifted invariants of sufficiently high order.  This is the essence of the (equivariant) moving frame method, \cite{OP-2008}.} to
signify its dependence on the pseudo-group variables $\psi_0\n =
(X_x,X_u,U_x,U_u,\ldots)$.

% , and let
% \[ \rho_n = \lrank f\n{}^* \hQ\n = \lrank \{
% \hQ_{ijk}(p,f^{(i+j+k)}(x,u,p),\psi^{(i+j+k+2)}: 0\leq i+j+k\leq n
% \} \] denote the joint rank of the lifted invariants of order $n$ or
% less.
% Let $\nu_n$ denote the dimension of the kernel of $d (f\n{}^*
% \hQ\n)$ in the vertical direction parameterized by the gauge variables
% $X_{ij}, U_{ij}$. We will call $\nu_n$ the $n$th-order isotropy
% % rank. 
% Again, because of the covariance properties of $\hQ_{ijk}$, the
% integer $\rho_n$ depends on $x,u,p$ only.
% $n^\text{th}$-order diffeomorphism jet $\psi\n$ has
% \[ 2\times(2+3+\ldots + n) = n(n+1)-2= (n+2)(n-1) \] components.  With
% this remark in mind, we set
% \begin{equation}
%   \label{eq:taundef} \tau_n := \rho_n + \nu_n - (n+4)(n+1)
% \end{equation} and refer to $\tau_n$ as the $n$th-order
% \emph{horizontal rank}.

The above remarks lead us to the following definition, after which we
will be ready to state our main result.
\begin{definition} 
  \label{def:regular} 
  We say that a smoothly defined ODE $q=f(x,u,p)$ is \emph{regular} if
  for every $n\geq 0$ the rank of the lifted invariants $\{ Q_{ijk}
  \colon 0\leq i+j+k\leq n \}$ is constant on the differential equation $q=f(x,u,p)$.
\end{definition}

\begin{remark}
Irregular differential equations that fail to satisfy Definition \ref{def:regular} are more difficult to classify and require more care, \cite{O-1995}.  Following customary practice, these equations are omitted in this paper.
\end{remark}

We now can state the main result of this paper.

\begin{theorem}\label{thm:main}
  Every regular second-order ODE is classified, relative to point
  transformations, by its 10th order jets.  This bound is sharp ---
  meaning that there exist regular ODEs that are not classified by 9th
  order jets.  Furthermore, every ODE having the maximal IC order of
  10 is equivalent, modulo point transformations and Cartan duality
  (see Appendix \ref{duality appendix} for the definition), to an
  equation of the form
  \begin{equation}
    \label{eq:ord4form}
     u_{xx} = 6u^2 + g(x) \quad\text{with}\quad D_{x}^2[g(x)^{-1/4}] \neq 0. 
  \end{equation}
\end{theorem} 

%To prove Theorem \ref{thm:main}, we use the method of equivariant
%moving frames, \cite{OP-2008}, which is a theoretical reformulation of
%Cartan's moving frame theory.  

The proof of Theorem \ref{thm:main} boils down to identifying the
branch(es) of the equivalence problem where non-constant absolute
invariants appear as late and slowly as possible during the course of
Cartan's normalization procedure.  Taking advantage of the
\emph{universal recurrence relations}, we first give a proof
of Theorem \ref{thm:main} which does not require explicit 
coordinate expressions for the differential invariants. This is possible since
the universal recurrence relations can be written down knowing only the 
expression for the prolonged infinitesimal generators of the pseudo-group 
action and the choice of a cross-section.  Using the fact that there is a notion of
duality among second-order ordinary differential equations (see
Appendix \ref{duality appendix}) our conclusion is that there exist
two families of differential equations (dual to each other) that
achieve the maximal IC order.  By a generalization of Cartan's
Integration Theorem, \cite{BCGGG-1991}, we are then able to show that
one of the two families of differential equations depends on one
arbitrary function of the independent variable.  Finally, we integrate
the structure equations for the canonical coframe and derive form
\eqref{eq:ord4form}.

Second-order ordinary differential equations equivalent to
\eqref{eq:ord4form} by point transformations and Cartan duality admit
three fundamental absolute invariants $I_7, I_8, I_9$ of the indicated
order, and a tenth order invariant $I_{10}$ which is functionally
dependent on $I_9$.  Relative to the normal form \eqref{eq:ord4form},
these invariants can be expressed as
\begin{equation}
  \label{eq:I7-10def}
  I_7 = \frac{p}{u^{3/2}},\qquad 
  I_8=\frac{g(x)}{u^2},\qquad
  I_9=\frac{(g^\prime(x))^4}{(g(x))^5}, \qquad 
  I_{10}=\frac{g(x) g''(x)}{(g'(x))^2}.  
\end{equation}
The functional relation between $I_9$ and $I_{10}$ is the essential
classifying relation for equations of maximal IC order.

% The pseudo-group of point transformation acts transitively on jets of
% order $\leq 4$.  At fourth order, the equivalence problem branches
% into four classes: the linearizable branch IV characterized by the
% vanishing of the 4th order relative invariants; branches II and III
% characterized by the vanishing of one, but not both of the invariants;
% and the generic branch IV characterized by the non-vanishing of both
% relative invaraints.  Class II and class III are related to one
% another by Cartan duality.

The class of equations \eqref{eq:ord4form} requiring 10th order jets
for their classification includes the Painlev\'e--I equation as the
subclass when $g(x)$ is linear in $x$.   
As a Corollary to Theorem \ref{thm:main} we
are able to give the following characterization of
Painlev\'e--I.
\begin{theorem}
  \label{thm:PI}
  The equivalence class of the Painlev\'e--I equation can be
  characterized as the subclass of second-order ODEs requiring 10th
  order jets for their classification and satisfying
  \[ q_{pppp}=0,\qquad I_{10} = 0. \]
\end{theorem}
\noindent
The condition $q_{pppp}=0$ distinguishes the Painlev\'e--I equation
from it's Cartan dual. The vanishing of $I_{10}=0$ means that the
Painlev\'e--I equation can be characterized as the ``simplest''
second-order differential equation belonging to the class of maximal
IC order equations.  We will derive \eqref{eq:I7-10def} at the end of
Section \ref{sect:integration}.  Thereafter Theorem \ref{thm:PI}
follows as a straight-forward Corollary of Theorem \ref{thm:main}.

\begin{remark}
An invariant characterization of the Painlev\'e--I and II equations up to fibre
preserving transformations and point transformations were given in \cite{KLS-1985,K-1989}.  In the latter case, the characterization obtained is a particular case of our more general result when $g(x)=x$.  Other works devoted to the Painlev\'eve--I and II equations can be found in \cite{BB-1999,D-2009,K-2009-0,K-2013}.

%Also, we note that the upper bound on the maximal IC order obtained in Theorem \ref{thm:main} cannot be deduced from \cite{K-1989} without additional work.
\end{remark}

\section{The geometric setting}\label{section:geosetting}
%%%%%
% Let $N^3 = \J^1(\R,\R)$ and $M^4=\J^2(\R,\R)$ denote, respectively,
% the first- and second-order jet space of curves $u=u(x)$.  The
% canonical coordinates on these manifolds are $x,u, p$ and $x,u,p,
% q$, respectively.  A second-order ODE \eqref{eq:uxx=q}
% corresponds to a section of the fibre bundle $M^4 \to N^3$.  Cartan
% showed that a second-order ODE gives rise to a canonically defined
% 8-dimensional principal bundle $B^8\to N^3$ equipped with an
% invariant coframe.  The structure functions of the invariant coframe
% and their derivatives constitute the signature manifold of the
% differential equations.  Two given equations are locally equivalent if
% and only if their signatures overlap on an open set.  The definition
% of the invariant coframe and Cartan's approach to the equivalence
% problem are described in \cite[Chapter 12, p. 403]{O-1995}.  

%In order to give an abstract proof of Theorem \ref{thm:main}, 
%\subsection{The groupoid definition}
In this section we introduce the geometric setting for the
equivalence problem of second-order ordinary differential equations under point transformations.
The key formalism is a certain groupoid and two sets of fundamental
equations: the \emph{universal recurrence relations} for the prolonged
jet coordinates, and the \emph{Maurer--Cartan structure equations} of
the infinite-dimensional Lie pseudo-group $\D =
\operatorname{Diff}(\R^2)$ \cite{OP-2005, OP-2008}.

Let $N^3 = \rJ^1(\R,\R)$ and $M^4=\rJ^2(\R,\R)$ denote, respectively,
the first- and second-order jet space of curves $u=u(x)$.  Setting $p=u_x$ and $q=u_{xx}$, local coordinates are given by
$$
N^3\colon x,\, u,\, p;\qquad M^4\colon x,\, u,\, p,\, q.
$$
A smoothly defined second-order ordinary differential equation
$q=f(x,u,p)$ is then identified as a smooth section $f\colon N^3 \to M^4$.  Let $\Gamma(N^3,M^4)$ denote the space of sections $f\colon N^3 \to M^4$, then $\J\n=\rJ^n \Gamma(N^3,M^4)$ can be identified as the bundle of $n\text{th}$-order jets of
second-order ODEs, \cite{O-1993}.  Local coordinates are given by
$$
\J\n\colon x,\,u,\,p,\, q\n;
$$
where $q\n=(\ldots\, q_{x^i u^j p^k}\, \ldots)$ collects the derivative coordinates of order $\leq n$.
The prolongation of a differential equation $q=f(x,u,p)$ yields a section of $\J\n$ which we denoted $f\n\colon N^3 \to \J\n$.  Next,
let $\G\subset \Diff(\R^4)$ denote the prolongation of $\Diff(\R^2)$
to $M^4$.  In local coordinates, the pseudo-group $\G$ is given by \eqref{eq:ptxform} and \eqref{eq:PQdef}.  
Geometrically, $\G$ specifies how a second-order ordinary differential equation transforms under a point transformation.  
For $0\leq n \leq \infty$, let $\D\n \to \R^2$ denote the
bundle of $n\text{th}$ order diffeomorphism jets of $\D=\Diff(\R^2)$ and similarly let
$\G\n$ denote the $n\text{th}$ order pseudo-group jet bundle of $\G$.  Local coordinates are given by
$$
\D\n\colon x,u,X,U,\psi_0\n,
$$
where $\psi_0\n = (\ldots\, X_{x^i u^j},\, U_{x^i u^j}\, \ldots)$ denotes the derivative coordinates of order 1 up to $n$.  Now,
let $\E\n \to \J\n$ denote the $n\text{th}$ order \emph{lifted bundle}
obtained by pulling back $\D^{(n+2)}\cong\G\n \to \R^2$ via the
projection $ \J\n \to M^4\to N^3\to \R^2$, \cite{OP-2008}.  Canonical bundle coordinates on $\E\n$ are
$$
\E\n\colon x,u,p,q\n, X,U,\psi_0^{(n+2)}.
$$

We note that the $n\text{th}$-order diffeomorphism jet
$\psi\n=(X,U,\psi\n_0)=(X,U,X_x,X_u,$ $U_x,U_u,\ldots)$ has
\[ 2\times(1+2+3+\ldots + n+1) = (n+1)(n+2) \] components and that the
$n\text{th}$-order jet $q\n=(q,q_x,q_u,q_p,\ldots) \in \J\n|_{(x,u,p)}$ has
$\binom{n+3}{3}$ components.   Hence
\begin{gather*}
%  \label{eq:dimJD}
  \dim \J\n = 3+\binom{n+3}{3},\qquad \dim \D\n = 2+(n+1)(n+2),\\
%  \label{eq:dimEEf}
  \dim \E\n = 3+\binom{n+3}{3} + (n+3)(n+4).
\end{gather*}

Next, we define the groupoid structure of the lifted bundle $\E\n$.
The source map $\bsn\colon\E\n\to \J\n$ is the standard projection given by
$(x,u,p,q^{(n)})$; while the target map $\bt\n\colon \mathcal{E}\n \to \J^n$ is
the projection given by the prolonged action $\left(X,U,
  P,Q\n\right)$, where
\[P = \hP(p,\psi_0^{(1)}) \] as per \eqref{eq:PQdef} and
\[ Q\n = \hQ\n(p,q\n, \psi_0^{(n+2)}) =  \left\{
\hQ_{ijk}(p,q\n,\psi_0^{(i+j+k+2)}): 0\leq i+j+k\leq n\right\}.\]

The diffeomorphism pseudo-group $\D=\Diff(\R^2)$ has a prolonged
action on $\J\n$ and two dual prolonged actions on $\E\n$. Given a
point transformation $(X,U)=\varphi(x,u)\in \D$ we define the prolonged
actions
\begin{equation*}
  \varphi\n \colon \J\n \to \J\n,\qquad \varphi_L\n \colon \E\n \to
  \E\n,\qquad 
  \varphi_R\n\colon\E\n \to \E\n 
\end{equation*}
according to  \cite{OP-2005}:
\begin{equation}
  \label{eq:phindef}
\begin{aligned}
  &
  \varphi\n\colon(x,u,p,q\n) \mapsto (X,U,\hP(p,\varphi_0^{(1)}),
  \hQ\n(p,q\n,\varphi_0^{(n+2)}))\,;\\ 
%  \label{eq:phildef}
  &\varphi_L\n\colon(x,u,p,q\n,\psi\n) \mapsto
  (x,u,p,q\n,(\varphi\circ\psi)\n)\,;\\
%  \label{eq:phirdef}
  &\varphi_R\n\colon(x,u,p,q\n,\psi\n) \mapsto
  (X,U,\hP(p,\varphi_0^{(1)}), \hQ\n(p,q\n,\varphi_0^{(n+2)}),
  (\psi\circ\varphi^{-1})\n )\,.
\end{aligned}
\end{equation}
From the above definitions, it follows immediately that the source
projection $\bsn$ is $\varphi_L$-invariant and $\varphi_R$-equivariant:
\begin{equation*}
%  \label{eq:sinvequiv}
  \bsn\circ \varphi_L\n = \bsn,\qquad \bsn\circ \varphi_R\n = \varphi\n \circ
  \bsn\,;
\end{equation*}
and that, dually, the target projection is $\varphi_R$-invariant and
$\varphi_L$-equivariant: 
\begin{equation}
  \label{eq:tinvequiv}
  \btn\circ \varphi_R\n = \btn,\qquad \btn\circ \varphi_L\n = \varphi\n \circ
  \btn\,. 
\end{equation}

%\section{The invariant coframe}
A coframe on $\J\ii$ is given by the basic horizontal one-forms
\begin{subequations}\label{eq:Jcoframe}
\begin{equation}
  dx,\qquad du,\qquad dp, 
\end{equation}
and the contact one-forms
\begin{equation}
\theta_{ijk}= dq_{ijk} - q_{i+1,j,k}\, dx - q_{i,j+1,k} \,du -
  q_{i,j,k+1}\, dp,\qquad i,j,k \geq 0.
\end{equation}
\end{subequations}
The standard coframe on $\D\ii$ is spanned by $dx, du$ together with
the group forms
\begin{equation}
  \label{eq:Upsilondef}
  \Upsilon_{ij} = dX_{ij} - X_{i+1,j}\, dx - X_{i,j+1}\, du,\quad
 \Psi_{ij} = dU_{ij} - U_{i+1,j}\, dx - U_{i,j+1}\, du,\quad i,j\geq0,
\end{equation}
while a right-invariant coframe on $\D\ii$ consists of the one-forms
\begin{equation}
  \label{eq:omxudef}
  \omega^x = X_x \,dx + X_u\, du,\qquad
  \omega^u = U_x\, dx + U_u \,du 
\end{equation}
and the Maurer--Cartan one-forms
\begin{equation}\label{eq:mcforms}
\mu_{ij}= \mu_{X^i U^j},\qquad \nu_{ij}=\nu_{X^i U^j},\qquad i,j \geq 0.
\end{equation}
The latter are defined, implicitly, by taking formal derivatives of
the relations
\[ \Upsilon= \mu,\qquad \Psi= \nu, \] with respect to $x,u$,
and then solving for the $\mu_{ij}, \nu_{ij}$.  The first few
relations that result are shown below:
\begin{align*}
  \Upsilon_{x} &= X_x\, \mu_X + U_x\, \mu_U, \qquad \Psi_{x} =
  X_x\, \nu_X  + U_x\, \nu_U ,\\
  \Upsilon_{u} &= X_u\, \mu_X + U_u\, \mu_U, \qquad
  \Psi_{u} =  X_u\, \nu_X  + U_u\, \nu_U ,\\
  \Upsilon^1_{xx} &= X_x^2\,\mu_{XX} + 2X_x U_x \,\mu_{UX} +U_x^2\,
  \mu_{UU} +
  X_{xx}\, \mu_X + U_{xx}\, \mu_U,\\
  \Upsilon^1_{ux} &= X_xX_u\,\mu_{XX} + (X_x U_u+X_u U_x) \mu_{UX}
  +U_xU_u\,\mu_{UU} +
  X_{ux}\, \mu_X + U_{ux}\, \mu_U,\\
  \Upsilon^1_{uu} &= X_u^2\,\mu_{XX} + 2X_u U_u\, \mu_{UX} +U_u^2 \,\mu_{UU}
  +
  X_{uu} \,\mu_X + U_{uu}\, \mu_U,\\
  &\hskip 0.22cm\vdots
\end{align*}
with the coefficients of the higher-order relations given by the
multi-variate F\`aa-di-Bruno polynomials, \cite{R-1980}.  The precise
coordinate expression of the Maurer--Cartan forms is derived in
\cite{OP-2005}, but these are not necessary for the symbolic
implementation of the moving frame method.

To construct an invariant coframe on $\E\ii$ we note that the space of
differential forms $\bO^*=\bO^*(\mathcal{E}\ii)$ on the infinite-order
lifted bundle decomposes into
$$
\bO^* = \oplus_{k,l} \, \bO^{k,l},
$$
where $k$ indicates the number of jet forms \eqref{eq:Jcoframe} and
$l$ the number of group forms \eqref{eq:Upsilondef}.  Let $\bO^*_J =
\oplus_k \,\bO^{k,0}$ denote the subspace of jet forms, and define the
projection map $\pi_J\colon \bO^* \to \bO^*_J$ onto the jet component.

\begin{definition}
  The \emph{lift} transformation $\bl\colon \bO^*(\Ji) \to \bO^*_J$ is
  defined by
  \begin{equation}
    \label{eq:lift}
    \bl = \pi_J \comp (\bt^{(\infty)})^*.     
  \end{equation}
\end{definition}
\noindent
Thus, by construction, the lift of a jet form on $\J\ii$ is an
invariant jet form defined on $\E\ii$.  
In particular, we have
\[   \omega^x = \bl(dx),\qquad \omega^u = \bl(du) . \]
We also introduce the invariant one-forms
\begin{equation}
  \label{eq:ompdef}
  \begin{aligned}
    \omega^p &= \bl(dp) = \pi_J (dP) \\
    & = (p X_u + X_x)^{-2} \Big\{(X_x U_u -X_u U_x) dp + \\
     &\quad \left.+ \Big(p^2 (U_{ux} X_u - X_{ux} U_u) + p
      (U_{xx}X_u- X_{xx}U_u +U_{ux} X_x- X_{ux}U_x )\right) + \\
     & \quad+(U_{xx}X_x - X_{xx} U_x)\Big) dx\\  &
    \quad\left.+ \Big(p^2 (U_{uu} X_u - X_{uu} U_u) + p (U_{uu}X_x-
      X_{uu}U_x +U_{ux} X_u- X_{ux}U_u )\right) + \\ 
    & \quad+(U_{ux}X_x - X_{ux} U_x)\Big) du\Big\},\\
    \vartheta_{0,0,0} &= \bl(\theta_{0,0,0}) = \frac{X_x U_u - U_x
      X_u}{(p X_u + X_x)^3} \left( dq - q_x dx - q_u du - q_p dp
    \right),
\end{aligned}
\end{equation}
and more generally,
\[ \vartheta_{ijk}=\bl(\theta_{ijk}). \]
%\subsection{The $\E\n$ structure equations}
% %
% \begin{equation}\label{invariant total derivative operators}
% D_X,\qquad D_U,\qquad D_P,
% \end{equation}
% %
% be the invariant total derivative operators dual to the lifted
% horizontal one-forms $\,\omega^x$, $\,\omega^u$, $\,\omega^p$.  
Next, we introduce the infinitesimal generator
\begin{equation}\label{infinitesimal generator}
\begin{aligned}
  \vv &= \xi^1(x,u) \pp{}{x} + \xi^2(x,u) \pp{}{u} + \xi^3(x,u,p) \pp{}{p} + \phi(x,u,p,q) \pp{}{q} + \sum_{i+j+k \geq 1} \phi^{ijk} \pp{}{q_{ijk}}\\
  \vv &= \xi(x,u) \pp{}{x} + \eta(x,u)\pp{}{u} + [\eta_x + p(\eta_u - \xi_x) - p^2 \xi_u]\pp{}{p} + [\eta_{xx} + q(\eta_u-2\xi_x) \\
  &\hskip 0.25cm+ p(2 \eta_{xu}-\xi_{xx})-3pq\xi_u +
  p^2(\eta_{uu}-2\xi_{xu})-p^3 \xi_{uu}] \pp{}{q} + \sum_{i+j+k \geq
    1} \phi^{ijk} \pp{}{q_{ijk}}
\end{aligned}
\end{equation} 
of the $\Diff(\R^2)$ action on $\J\ii$ obtained by prolonging
\eqref{eq:Qtarget}.  The coefficients $\phi^{ijk}$ are defined
  recursively by the usual prolongation formula
\begin{equation}\label{prolongation formula}
  \phi^{ijk} = D_x^i D_u^j D_p^k(\phi-\xi^1 q_x - \xi^2 q_u - \xi^3
  q_p) + \xi^1 q_{ij,k+1} + \xi^2 q_{i,j+1,k} + \xi^3
  q_{i,j,k+1},
\end{equation}
where
\begin{gather*}
  D_x = \frac{\partial}{\partial x} + \sum_{i,j,k\geq 0}  q_{{i+1},jk}
  \frac{\partial}{\partial q_{ijk}},\quad
  D_u = \frac{\partial}{\partial u} + \sum_{i,j,k\geq 0} q_{i,j+1,k}
  \frac{\partial}{\partial q_{ijk}},\\
  D_p = \frac{\partial}{\partial p} + \sum_{i,j,k\geq 0} q_{ij,k+1}
  \frac{\partial}{\partial q_{ijk}},
\end{gather*}
are the total derivative operators on $\J\ii$.  

We now extend of the definition of the lift map \eqref{eq:lift} to the vector field jet coordinates $\xi_{ij}$ and $\eta_{ij}$ following \cite[Section 5]{OP-2005}.  For this, let 
\begin{equation}\label{eq:v}
\vv = \xi(x,u)\pp{}{x} + \eta(x,u)\pp{}{u}
\end{equation}
be an infinitesimal generator of $\D = \text{Diff}(\R^2)$.  Then the lift of \eqref{eq:v} is the right-invariant vector field
$$
\bl(\vv) = \sum_{i,j\geq 0} \bigg[\mathbb{D}^i_x\mathbb{D}^j_u\xi(X,U) \pp{}{X_{ij}} + \mathbb{D}^i_x\mathbb{D}^j_u\eta(X,U) \pp{}{U_{ij}}\bigg]
$$
tangent to the source fibers of $\D^{(\infty)}$, where
$$
\mathbb{D}_x = X_x D_X + U_x D_U,\qquad \mathbb{D}_u = X_u D_X + U_u D_U.
$$
Let $\ji\vv \in \rJ^\infty T\mathbb{R}^2$ denotes the infinite jet of \eqref{eq:v}, then the lift of a section $\zeta \in (\rJ^\infty T\mathbb{R}^2)^*$ in the dual bundle to the vector field jet bundle $\rJ^\infty T\mathbb{R}^2$  is defined by the equality
\begin{equation}\label{eq:liftsection}
\langle \bl(\zeta); \bl(\vv)\rangle\big|_{\phi\ii} = \langle \zeta; \ji \vv \rangle\big|_{(X,U)}\qquad \text{whenever}\qquad 
\begin{aligned}
&\phi\ii \in \D\ii, \\ &(X,U) = \bt(\phi).
\end{aligned}
\end{equation}
Since each vector field jet coordinate functions $\xi_{x^iu^j}=\xi_{ij}$, $\eta_{x^iu^j}=\eta_{ij}$ can be viewed as sections of $(\rJ^\infty T\mathbb{R}^2)^*$ we have from \eqref{eq:liftsection} the following defining equalities
\begin{equation}\label{eq:mcforms}
\bl(\xi_{ij})\colon\hskip -0.25cm=\mu_{ij},\qquad
\bl(\eta_{ij})\colon\hskip -0.25cm=\nu_{ij},\qquad i,j \geq 0, 
\end{equation}
where $\mu_{ij}$, $\nu_{ij}$ are the Maurer--Cartan forms introduced in \eqref{eq:mcforms}.  %The fact that the lift of a vector field jet coordinate equals a differential form might appear strange.  But as explained below Proposition 5.3 of \cite{OP-2008}, at each point, the vector field jet coordinates $\xi_{ij}$, $\eta_{ij}$ define linear functions on the space of vectors $\mathcal{X}(\mathbb{R}^2)$ and can therefore be thought as some differential forms. 

We also note that the lift of the source variables, called {\it lifted invariants}, gives the
target variables:
\begin{equation}\label{lifted invariants}
X=\bl(x),\qquad U=\bl(u),\qquad P=\bl(p),\qquad Q_{ijk}=\bl(q_{ijk}).
\end{equation}

\begin{proposition}\label{universal recurrence relations proposition}
  The \emph{universal recurrence relations} for the lifted invariants
  % $X$, $U$, $P$, $Q_{ijk}$
are
  \begin{subequations}
    \label{eq:unirec}
% \label{universal recurrence relations}
    \begin{align}
      \label{eq:XUPrecurrence}
      &\begin{aligned}
      &  dX = \omega^x+ \mu,\\ 
      & dU = \omega^u+\nu,\\ 
      & dP = \omega^p+\nu_X+P(\nu_U-\mu_X) - P^2\mu_U,
      \end{aligned}\\
        &\hskip -0.4cm dQ_{ijk} = Q_{i+1,jk}\, \omega^x +
      Q_{i,j+1,k}\,\omega^u +Q_{ij,k+1}\, \omega^p+ 
      \vartheta_{ijk}+\bl(\phi^{ijk}),\label{eq:genQrec}
      \end{align}
\end{subequations}
where $\bl(\phi^{ijk})$ is the lift of the prolonged vector field coefficient \eqref{prolongation formula}.
\end{proposition}

By the prolongation formula \eqref{prolongation formula} and \eqref{infinitesimal generator}, the vector field coefficient $\phi^{ijk}$ is well-defined linear combination of the vector field jet coordinates $\xi_{ij}$, $\eta_{ij}$ with polynomial coefficients in $p$ and $q_{ijk}$.  Thus, by virtue of \eqref{eq:mcforms} and \eqref{lifted invariants} the correction term $\bl(\phi^{ijk})$ in \eqref{eq:genQrec} is a certain linear combination of the Maurer--Cartan forms $\mu_{ij}$, $\nu_{ij}$ whose coefficients depend polynomially on the lifted invariants $P$, $Q_{ijk}$.

Finally, in our analysis we will need to consider the structure equations of the Maurer--Cartan forms \eqref{eq:mcforms}.

\begin{proposition} The \emph{Maurer--Cartan structure equations} of
  the diffeomorphism pseudo-group \eqref{eq:ptxform} are,
  \cite{OPV-2009},
\begin{equation}\label{eq:gendiffmc}
\begin{aligned} 
  d\mu_{ij} &= \hskip-12pt\sum_{\substack{(0,0)\leq (k,\ell)\leq
      (i,j)\\(k,\ell)\neq (i,j)}}
  \binom{i}{k}\binom{j}{\ell}\big(\mu_{k+1,\ell} \wedge
  \mu_{i-k,j-\ell} + \mu_{k,\ell+1} \wedge \nu_{i-k,j-\ell}\big), \\
  &\quad - \mu_{i+1,j} \wedge \omega^x -
  \mu_{i,j+1}\wedge \omega^u,\\
  d \nu_{ij} &= \hskip-12pt\sum_{\substack{(0,0)\leq (k,\ell)\leq
      (i,j)\\(k,\ell)\neq (i,j)}}
  \binom{i}{k}\binom{j}{\ell}\big(\nu_{k+1,\ell} \wedge
  \mu_{i-k,j-\ell} + \nu_{k,\ell+1} \wedge \nu_{i-k,j-\ell}\big)\\
  &\quad - \nu_{i+1,j} \wedge \omega^x - \nu_{i,j+1}\wedge \omega^u .
\end{aligned}
\end{equation}
\end{proposition}

%%%%%

\section{The equivalence problem}

With all the tools in hand, we can now delve into the point equivalence problem of second-order ODEs and prove Theorem \ref{thm:main}.  We begin by setting the equivalence problem within the geometrical framework of the previous section.

\subsection{The direct, infinite-dimensional formulation}
Given a smoothly defined ODE $q=f(x,u,p)$, let $\E_f\n\to N^3$ denote
the bundle over $N^3$ given by the pullback of $\E\n$ by $f\n$.  The
canonical bundle coordinates on $\E_f\n$ are
$x,u,p,X,U,\psi_0^{(n+2)}$, which means that
\[ \dim \E_f\n = 3+(n+3)(n+4). \] The corresponding embedding $\E_f\n
\hookrightarrow \E\n$ is given by
\[ (x,u,p,X,U,\psi_0^{(n+2)}) \mapsto (x,u,p,f\n(x,u,p),X,U,
\psi_0^{(n+2)}).\] It follows that $\E_f\n$ is no longer a groupoid,
but merely a fibre bundle. The corresponding restrictions of the
source and target projections to $\E_f\n$, denoted $\bsn_f \colon
\E_f\n\to N^3$ and $\btn_f\colon\E_f\n \to \J\n$, respectively, are
given by
\begin{align*}
  &\bsn_f \colon (x,u,p,X,U,\psi_0^{(n+2)}) \mapsto (x,u,p),\\
  &\btn_f \colon (x,u,p,X,U,\psi_0^{(n+2)}) \mapsto
  (X,U,\hP(p,\psi_0^{(1)}) , \hQ\n(p,f\n(x,u,p),\psi_0^{(n+2)})).
\end{align*}
Since the mapping \[(p,\psi\n) \mapsto (X,U,\hP(p,\psi_0^{(1)}))\] has
constant rank, the rank of $\btn_f$ is constant if and only if the ODE
is regular as per Definition \ref{def:regular}.  Also, by
\eqref{eq:tinvequiv}, the restricted target map, $\bt\n_f$, is
invariant with respect to point transformations.  Hence, if a given
ODE is classified by $n$th order jets, then the $\img\btn_f\subset
\J\n$ serves as a signature manifold for the ODE; that is, two ODEs are
locally point-equivalent if and only if their signatures overlap on an
open set.

The contact invariant one-forms $\vartheta_{ijk}$ in
\eqref{eq:genQrec} are essential when working with the invariant
variational bicomplex, \cite{KO-2003,TV-2011}, or studying geometric
submanifold flows, \cite{MO-2010,O-2008}.  Since the restriction of
$\E\n$ to $\E_f\n$ annihilates the one-forms $\vartheta_{ijk}$, the
equivalence problem must be formulated in terms of the invariant
$\E\ii_f$-coframe
\begin{equation}
  \label{eq:Efcoframe}
  \omega^x,\qquad
  \omega^u,\qquad \omega^p,\qquad \mu_{ij},\qquad \nu_{ij}.
\end{equation}
The structure equations for
this coframe consist of \eqref{eq:gendiffmc} as well as
% From the above, we can derive structure equations of the
% canonical $\E\n$ coframe.  The first few relations are shown below;
% the general relations can be inferred from \eqref{eq:gendiffmc}.
\begin{equation}
\label{eq:domegaxup}
\begin{aligned} 
& d\omega^x= -d\mu = \mu_X\wedge\omega^x + \mu_U\wedge
\omega^u,\\
& d\omega^u=-d\nu = \nu_X\wedge\omega^x +
\nu_U\wedge\omega^u,\\
& d\omega^p = (\nu_U-\mu_X - 2P \mu_U)\wedge \omega^p +
(\nu_{UX}+P(\nu_{UU}-\mu_{UX}) -P^2 \mu_{UU})\wedge \omega^u
\\ &\qquad + (\nu_{XX}+P(\nu_{UX}-\mu_{XX}) -P^2
\mu_{UX})\wedge \omega^x. \\ 
% & d\mu_X =
% \omega^x\wedge\mu_{XX} + \omega^u\wedge\mu_{UX} + \mu_U \wedge
% \nu_X,\nonumber\\
% %
% & d\mu_U =\omega^x\wedge\mu_{UX} + \omega^u\wedge\mu_{UU} + \mu_U
% \wedge (\nu_U-\mu_X),\nonumber\\ & d\mu_{XX} = \omega^x\wedge\mu_{XXX}
% + \omega^u\wedge\mu_{UXX} + 2 \mu_{UX}\wedge\nu_X - \nu_{XX} \wedge
% \mu_U + \mu_{XX} \wedge \mu_{X},\nonumber\\
% %
% &\hskip 1cm\vdots\label{mc structure equations}\\
% %
% &d\nu_X = \omega^x\wedge\nu_{XX} + \omega^u\wedge\nu_{UX} + \nu_X
% \wedge (\mu_X- \nu_U),\nonumber\\
% %
% &d\nu_U = \omega^x\wedge\nu_{UX} + \omega^u\wedge\nu_{UU} + \nu_X
% \wedge \mu_U ,\nonumber\\
% %
% &d\nu_{XX}= \omega^x\wedge\nu_{XXX} + \omega^u\wedge\nu_{UXX}+\nu_{XX}
% \wedge (2\mu_X-\nu_U) + 2 \nu_{UX} \wedge \nu_X - \mu_{XX}\wedge \nu_X
% ,\nonumber\\
% %
% &d\nu_{UX}=\omega^x\wedge\nu_{UXX} + \omega^u\wedge\nu_{UUX}+ \nu_{XU}
% \wedge \mu_X + \nu_{XX} \wedge \mu_U- \mu_{XU}\wedge \nu_X + \nu_{UU}
% \wedge \nu_X,\nonumber\\
% %
% &d\nu_{UU}= \omega^x\wedge\nu_{UUX} + \omega^u\wedge\nu_{UUU}+ 2
% \nu_{UX} \wedge \nu_U - \mu_{UU}\wedge \nu_X + \nu_{UU} \wedge
% \nu_U,\nonumber\\
% %
% &\hskip 1cm\vdots.\nonumber
\end{aligned}
\end{equation}
The latter are obtained by computing the exterior derivative of
\eqref{eq:XUPrecurrence}, and taking into account the Maurer--Cartan
structure equations \eqref{eq:gendiffmc}.

On $\E\ii_f$, the universal recurrence relations \eqref{eq:genQrec}
must be considered modulo the contact one-forms
$\bvtheta=\{\vartheta_{ijk}\}$.  These relations then express the
one-forms $dQ_{ijk}$ as invariant linear combinations of the coframe
\eqref{eq:Efcoframe}.  So in effect, the present setting can be
considered as an infinite-dimensional overdetermined equivalence
problem; see \cite[p. 297]{O-1995} for a discussion.

\subsection{The universal reduction}\label{section:ureduction}
The direct formulation of the equivalence problem outlined in the
preceding subsection suffers from an essential difficulty stemming
from the fact that the lifted invariants depend on an unbounded number
of pseudo-group variables.  Indeed, there is no upper bound on
$\lrank\btn_f$ as $n\to \infty$, and so, apriori, it is not even
possible to assert that an IC bound exists.

To overcome this difficulty, we introduce a partial moving frame
\cite{O-2011,V-2012} for the action of the point transformation
pseudo-group on $\J\ii$.  In effect, this moves the equivalence
problem from an infinite-dimensional setting to an 8-dimensional
principal bundle. After this \emph{universal} reduction, which is
valid for all smoothly defined ODEs, the invariant classification
proceeds using the usual method of reduction of structure, \cite{K-1989,O-1995}, albeit with
a certain amount of branching.

To illustrate the normalization procedure, we first consider the order
zero normalization in some details, and then pass to the description
of the full normalization.  Considering the recurrence relations
\eqref{eq:XUPrecurrence} and
\begin{align*} dQ &= Q_X \,\omega^x+ Q_U \,\omega^u + Q_P \,\omega^p +
  \bl(\phi)\\
  &= Q_X \,\omega^x+ Q_U \,\omega^u + Q_P \,\omega^p +
  \nu_{XX}+Q(\nu_U-2\mu_X)+P(2\nu_{XU}-\mu_{XX})\\
  &\quad-3PQ\mu_U+P^2(\nu_{UU}-2\mu_{XU})-P^3\mu_{UU},
\end{align*}
we see that it is possible to normalize
\begin{equation}
  \label{eq:0ordnorm} 
  X,U,P,Q\to0
\end{equation}
as their exterior derivatives involve the linearly independent
Maurer--Cartan forms $\mu$, $\nu$, $\nu_X$, $\nu_{XX}$. The result is
the system of equations
$$
0=\omega^x + \mu,\qquad 0=\omega^u + \nu,\qquad 0=\omega^p +
\nu_X,\qquad 0 = Q_P \,\omega^p + Q_U \,\omega^u + Q_X \,\omega^x +
\nu_{XX},
$$
which can be solved for the partially normalized Maurer--Cartan forms:
$$
\mu=-\omega^x,\qquad \nu=-\omega^u,\qquad \nu_X=-\omega^p,\qquad
\nu_{XX} = -(Q_P \,\omega^p + Q_U \,\omega^u + Q_X \,\omega^x).
$$
\begin{remark}\label{remark:gstructure}
After substituting \eqref{eq:0ordnorm} in
\eqref{eq:omxudef} and \eqref{eq:ompdef} and normalizing the
pseudo-group jets, we recover the usual $G$-structure formulation of
the equivalence problem \cite{O-1995},
\begin{equation}\label{G structure}
\begin{pmatrix} \omega^u\\ \omega^p \\ \omega^x
\end{pmatrix} =
\begin{pmatrix} a_1 & 0 & 0\\ a_2 & a_1/a_4 & 0\\ a_3 & 0 & a_4
\end{pmatrix}
\begin{pmatrix} du - p\, dx\\ dp - q\, dx\\ dx
\end{pmatrix},
\end{equation}
where
$$
a_1 = U_u,\qquad a_2 = \frac{\hD_x(U_u)}{\hD_x(X)},\qquad a_3 = X_u,\qquad
a_4 =\hD_x(X) = p X_u + X_x.
$$
\end{remark}

Continuing the normalization procedure, order by order, let
$\Xi\subset \J\ii$ be the submanifold defined by 
\begin{equation}  \label{eq:univnorm}
\begin{aligned}
  &x=u=p=0,\\ 
  &q_{ij0}= q_{ij1}  = 0,\quad i,j \geq 0, \\ 
  &q_{0j2} = q_{1j2} = q_{0j3} = q_{1j3} = 0,\quad j \geq 0,
\end{aligned}
\end{equation}
and let
\[\tE = (\bt^{(\infty)})^{-1}(\Xi)\subset \E\ii\]
denote the lift of $\Xi$; that is, $\tE$ is defined by the equations
\begin{equation}
  \begin{aligned}
    \label{eq:liftedunorms}
    &X=U=P=0,\\
    &Q_{ij0}= Q_{ij1}  = 0,\quad i,j \geq 0, \\
    &Q_{0j2} = Q_{1j2} = Q_{0j3} = Q_{1j3} = 0,\quad j \geq 0.
  \end{aligned}
\end{equation}
Let $H\subset \SL_3\R$ be the 5-dimensional subgroup 
\[ H = \left\{
  \begin{pmatrix}
    a_1 & a_2 & 0\\
    0 & b_2 & 0\\
    c_1 & c_2 & c_3
  \end{pmatrix}: a_1 b_2 c_3 = 1 \right\} \] of fractional linear
transformations that preserve the origin $(x,u)=(0,0)$.  The proof of
the next two Propositions is presented in Appendix
\ref{apdx:norm}. 

\begin{proposition}\label{prop:norm}
The submanifold $\Xi\subset\J\ii$ is a global\footnote{This means that any jet $(x,u,p,q\ii) \in \mathcal{J}\ii$ can be mapped to a point in $\Xi$ under the action of $\text{Diff}(\R^2)$.} cross-section for the action of $\Diff(\R^2)$ on $\J\ii$.
\end{proposition}

\begin{proposition}
  \label{prop:Hstab}
  We have $H = \G_\Xi$; that is, $H$ is the subgroup of $\SL_3\R$ that
  preserves $\Xi$.
\end{proposition}
\noindent 

In light of Proposition \ref{prop:norm}, we can utilize $\Xi$ as a
partially normalizing cross-section for the point-equivalence problem, \cite{O-2011}.
Given a smoothly defined ODE $q=f(x,u,p)$, let $\tbs_f\colon\tE_f \to
N^3$ denote the pullback bundle of $\tbs\colon\tE\to \J\ii$ via $f\ii \colon
N^3 \to \J\ii$.  By Proposition \ref{prop:Hstab}, this pullback is a
reduction of structure from $\bs_f\colon \E\ii_f \to \J\ii$ to the
principal $H$-bundle, $\tbs_f\colon\tE_f\to N^3$.

In the sequel, we use the tilde decoration to denote the pullback to
$\tE_f$, and refer to the quantities
\begin{equation*}
%  \label{eq:tQijkdef}
  \tQ_{ijk} = \hQ_{ijk}(p,q^{(i+j+k)},\psi^{(i+j+k+2)}_0)\Big|_{\tE_f}
\end{equation*}
as \emph{universal invariants}.  The universal invariants are, in
fact, the components of $\tbt_f\colon\tE_f \to \Xi$; the latter
obtained by imposing the normalizations \eqref{eq:liftedunorms}.  As
such, the universal invariants are functions of $3+5=8$ variables and
are $H$-equivariant with respect to the restricted left action
\eqref{eq:phindef}.

As is shown in Appendix \ref{apdx:norm}, the following one-forms
\[ \mu,\qquad \nu,\qquad\nu_X, \qquad \nu_{XX}, \qquad \mu_{ij},\; i+j\geq
2,\qquad \nu_{ij},\; i+j \geq 3, \] are normalized by
\eqref{eq:liftedunorms}.  In particular, applying
\eqref{eq:liftedunorms} to the universal recurrence formulas for
\[ dQ_{P^k}, \qquad dQ_{P^kX},\qquad dQ_{P^kU}, \qquad k=0,1,2,3,\]
yields following relations:
\begin{equation}
  \label{eq:1formnorm}
  \begin{aligned}
    &\tmu = -\tomega^x,\qquad \tnu = -\tomega^u,\qquad \tnu_X =
    -\tomega^p,\\
    &\tnu_{X^2} = \tnu_{X^3} = \tnu_{X^2U} = \tmu_{X^4} = 0,\qquad
    \tmu_{X^2} = 2 \tnu_{XU},\\
    &\tmu_{XU}=\frac{1}{2} \tnu_{UU},\qquad \tmu_{U^2} = \frac{1}{6}
    \tQ_{P^4} \tomega^p,\qquad 
    3\tmu_{X^2U} = 6 \tnu_{XU^2} = \tQ_{P^2X^2}\, \tomega^x,\\
    &6\tmu_{XU^2} = 3 \tnu_{U^3} = \tQ_{P^3X^2}\, \tomega^x+\tQ_{XP^4}\,
    \tomega^p,\qquad
    6\tmu_{U^3} = - \tQ_{P^4}\,\tnu_{XU} + \tQ_{P^4U}\, \tomega^p.
  \end{aligned}
\end{equation}
The expressions \eqref{eq:1formnorm} and subsequent formulas were obtained by importing the universal recurrence relations \eqref{eq:XUPrecurrence} into {\sc Mathematica}.  Substituting the above relations into \eqref{eq:gendiffmc} and \eqref{eq:domegaxup} yields the structure equations for $\tE_f$,
namely:

\begin{align}
&d\tomega^x = \tmu_X \wedge \tomega^x + \tmu_U \wedge \tomega^u,\nonumber\\
&d\tomega^u = \tnu_U \wedge \tomega^u + \tomega^x \wedge \tomega^p,\nonumber\\
&d\tomega^p = \tnu_{UX} \wedge \tomega^u + (\tnu_U - \tmu_X) \wedge
\tomega^p,\nonumber\\
&d\tmu_X = -2\,\tnu_{UX} \wedge \tomega^x - \frac{1}{2} \tnu_{UU}\wedge
\tomega^u - \tmu_U \wedge \tomega^p,\label{eq:uredseq}\\
&d\tmu_U = -\frac{1}{2}\tnu_{UU}\wedge \tomega^x + (\tmu_X -
\tnu_U)\wedge \tmu_U + \frac{1}{6} \tQ_{P^4} \tomega^u \wedge \tomega^p,\nonumber\\
&d\tnu_U = -\tnu_{UX} \wedge \tomega^x + \tmu_U \wedge \tomega^p -
\tnu_{UU}\wedge \tomega^u,\nonumber\\
&d\tnu_{UU} = 2 \tnu_{UX} \wedge \tmu_U + \tnu_{UU}\wedge \tnu_U +
\frac{1}{3} \tQ_{P^4X}\, \tomega^u\wedge \tomega^p +
\frac{1}{3}\tQ_{P^3X^2}\,\tomega^u \wedge \tomega^x,\nonumber\\
&d\tnu_{UX} = \tnu_{UX} \wedge \tmu_X - \frac{1}{2} \tnu_{UU} \wedge
\tomega^p + \frac{1}{6}\tQ_{P^2X^2}\, \tomega^u \wedge \tomega^x.\nonumber
\end{align}

We now  define the
\emph{reduced  rank sequence}
\begin{equation}
  \label{eq:trhodef}
  \tvarrho_n \colon\hskip-0.25cm= \lrank \tbt_f\n,\quad n\geq 0.
\end{equation}
\begin{proposition}
  \label{prop:3reg}
  The following are equivalent:
  \begin{itemize}
  \item[(i)] $q=f(x,u,p)$ is a regular ODE;
  \item[(ii)] for every $n$, the reduced rank  $\tvarrho_n$ is constant;
  \item[(iii)] the invariant coframe
    \begin{equation}\label{eq:8invcoframe}
    \tomega^x,\qquad \tomega^u,\qquad \tomega^p,\qquad \tmu_X,\qquad \tmu_U,\qquad \tnu_U,\qquad \tnu_{XU},\qquad \tnu_{UU} 
    \end{equation}
    on $\tE_f$ is fully regular \cite[Definition
    8.14]{O-1995}.
  \end{itemize}
\end{proposition}
\begin{proof}
  The equivalence of (i) and (ii) follows from the equivariance
  property \eqref{eq:tinvequiv} of the target projection.  The
  equivalence of (ii) and (iii) follows from the structure equations
  \eqref{eq:uredseq} and from the universal recurrence relations
  \eqref{eq:genQrec}.
\end{proof}
\noindent In light of the above remarks and Proposition
\ref{prop:3reg} we observe that
\[ \tvarrho_0= \tvarrho_1=\tvarrho_2=\tvarrho_3=0,\quad\text{and}\quad
\tvarrho_4\leq \tvarrho_5\leq \tvarrho_6 \leq \ldots \leq 8 .\] It
follows that the reduced rank sequence stabilizes at a sufficiently
high order $n$.  Indeed, the following is true.
\begin{proposition}
  The IC order can be characterized as the smallest integer $n\geq 4$
  such that $\tvarrho_{n-1} = \tvarrho_n$.
\end{proposition}
\begin{proof}
  See Proposition 8.18 of \cite{O-1995}.
\end{proof}

\begin{remark}\label{remark:estructure}
As alluded in Remark \ref{remark:gstructure}, the equivariant moving frame formalism offers an alternative approach to the $G$-structure formulation of the equivalence problem.  The 8-dimensional coframe \eqref{eq:8invcoframe} and the structure equations \eqref{eq:uredseq} obtained after carrying out the universal normalizations \eqref{eq:liftedunorms} can also be found using Cartan's equivalence algorithm, \cite{C-1955,K-1989,NS-2003,O-1995}.   At this stage, the equivalence problem splits into different branches according to the values of $\tQ_{P^4}$ and $\tQ_{P^2X^2}$.  As advocated by Gardner, \cite{G-1983, O-1995}, the different scenarios could be analyzed symbolically using the structure equations \eqref{eq:uredseq} and the identity $d^2=0$ for the exterior derivative.  But as mentioned in \cite{O-1995}, one has to be careful as one might be led down spurious branches of the equivalence problem owing to unexpected normalizations or cancellations due to the explicit forms of the coframe.  In the equivariant formalism we dispense ourself from these computations and issues by exploiting the recurrence relations \eqref{eq:genQrec}.

Another benefit of the moving frame formalism is the possibility of determining the order of an invariant without knowing its coordinate expression which is not something that can be easily done within Cartan's framework.  According to \cite[Lemma 7.4]{OP-2009}, once the pseudo-group action becomes free at order $n$, in other word all the pseudo-group parameters of the $n$th prolonged action can be normalized, then the normalization of a lifted invariant $Q_{ijk}$ of order $i+j+k \leq n$ is an invariant of order $i+j+k$.
\end{remark} 

%%%%%
\subsection{The fundamental branching}
\label{sect:branching}
%%%%%
The universal reduction \eqref{eq:liftedunorms} leads to the
normalization of all the Maurer--Cartan forms \eqref{eq:mcforms}
except for
\begin{equation}\label{unnormalized mc forms}
\tmu_X,\qquad \tmu_U,\qquad \tnu_U,\qquad \tnu_{UU},\qquad \tnu_{XU}.
\end{equation}
To proceed further, the value of the universal invariants 
\begin{equation}\label{intermediate remaining invariants}
  \tQ_{P^{k+4} U^j X^i},\qquad \tQ_{P^3 U^j X^{i+2}},\qquad \tQ_{P^2 U^j X^{i+2}},\qquad i,j,k\geq 0,
\end{equation}
(recall that these are restrictions of the lifted invariants $Q_{ijk}$
to the universal normalizing cross-section \eqref{eq:liftedunorms})
must be analyzed in more details order by order.  Up to order 6, the non-trivial
universal invariants \eqref{intermediate remaining invariants} are
\begin{equation}\label{456 invariants}
\begin{aligned}
n=4\colon&\qquad \tQ_{P^4},\; \tQ_{P^2 X^2},\\
n=5\colon&\qquad \tQ_{P^5},\; \tQ_{P^4 U},\; \tQ_{P^4 X},\; \tQ_{P^3 X^2},\; \tQ_{P^2 U X^2},\; \tQ_{P^2 X^3},\\
n=6\colon&\qquad \tQ_{P^6},\; \tQ_{P^5 U},\; \tQ_{P^5 X},\; \tQ_{P^4 U^2},\; \tQ_{P^4 U X},\; \tQ_{P^4 X^2},\; \tQ_{P^3 U X^2},\\
&\qquad \tQ_{P^3 X^3},\; \tQ_{P^2 U^2 X^2},\; \tQ_{P^2 U X^3},\; \tQ_{P^2 X^4}.
\end{aligned}
\end{equation}
Writing the recurrence relations for the invariants \eqref{456
  invariants} of order $\leq 5$ we obtain
\begin{align}
  d\tQ_{P^4} &= \tQ_{P^5} \,\tomega^p + \tQ_{P^4U} \,\tomega^u +
  \tQ_{P^4X}\,\tomega^x + \tQ_{P^4}(2 \tmu_X - 3 \tnu_U),\nonumber\\
  d\tQ_{P^2X^2} &= \tQ_{P^3 X^2}\,\tomega^p + \tQ_{P^2 U X^2}
  \,\tomega^u + \tQ_{P^2 X^3} \,\tomega^x - \tQ_{P^2X^2}(\tnu_U + 2
  \tmu_X),\nonumber\\
  d\tQ_{P^5} &= \tQ_{P^6} \,\tomega^p + \tQ_{P^5U} \,\tomega^u +
  \tQ_{P^5X}\,\tomega^x + 5 \tQ_{P^4} \,\tmu_U + \tQ_{P^5} (3 \tmu_X -
  4 \tnu_U),\nonumber\\ 
  d\tQ_{P^4X} &= (\tQ_{P^5 X} + \tQ_{P^4 U})\,\tomega^p + \tQ_{P^4 U
    X} \,\tomega^u + \tQ_{P^4 X^2}\,\tomega^x + \tQ_{P^4}\, \tnu_{UX}
  \nonumber\\   &\quad + \tQ_{P^4 X}(\tmu_X - 3 \tnu_U),\nonumber\\ 
  d\tQ_{P^4U} &= \tQ_{P^5 U} \,\tomega^p + \tQ_{P^4 U^2} \,\tomega^u +
  \tQ_{P^4 U X}\,\tomega^x - 2 \tQ_{P^4}\, \tnu_{UU} - \tQ_{P^5}\,
  \tnu_{UX} - \tQ_{P^4 X}\,\tmu_U\nonumber\\ 
  &\quad +\tQ_{P^4 U}(2 \tmu_X - 4 \tnu_U),\label{45 recurrence
    relations}\\ 
  d\tQ_{P^3X^2} &= \tQ_{P^4X^2}\, \tomega^p + \tQ_{P^3 U X^2}\,
  \tomega^u + (\tQ_{P^3 X^3} -2 \tQ_{P^2 U X^2}) \, \tomega^x -
  \tQ_{P^2X^2}\, \tmu_U \nonumber\\
  &\quad - \tQ_{P^3X^2}(2\tnu_U + \tmu_X),\nonumber\\
  d\tQ_{P^2 U X^2} &= \tQ_{P^3 U X^2}\, \tomega^p + \tQ_{P^2 U^2
    X^2}\, \tomega^u + \tQ_{P^2 U X^3}\, \tomega^x - 2 \tQ_{P^2X^2}\,
  \tnu_{UU} - \tQ_{P^3X^2}\, \tnu_{UX} \nonumber\\ 
  &\quad - \tQ_{P^2X^3} \,\tmu_U - 2 \tQ_{P^2 U X^2}(\tnu_U + \tmu_X),\nonumber\\
  d\tQ_{P^2X^3} &= (\tQ_{P^3 X^3}- \tQ_{P^2UX^2})\,\tomega^p +
  \tQ_{P^2 U X^3}\, \tomega^u + \tQ_{P^2 X^4}\, \tomega^x - 5
  \tQ_{P^2X^2}\,\tnu_{UX} \nonumber\\ 
  &\quad - \tQ_{P^2X^3}(\tnu_U + 3 \tmu_X).\nonumber
\end{align}
Considering the first two recurrence relations in \eqref{45 recurrence relations}, and concentrating on the correction terms involving the 
partially normalized Maurer--Cartan forms $\tmu_X$, $\tnu_U$, we notice that the values of $\tQ_{P^4}$ and $\tQ_{P^2X^2}$ will govern
the next possible normalizations.  For example, if  $\tQ_{P^4} \equiv \tQ_{P^2X^2} \equiv 0$ then the correction terms vanish and 
 $\tQ_{P^4}$, $\tQ_{P^2X^2}$ are genuine invariants that cannot be normalized. In this case, higher order universal invariants have
to be considered in order to normalize the remaining pseudo-group parameters.   In total, there are 4 different cases splitting the 
equivalence problem into 4 branches:
%
%At this juncture, the equivalence problem splits into 4 branches
%depending on the values of the fourth-order universal invariants:
% \footnote{To distinguish between lifted invariants that are set
% equal to zero by normalization from those that are assumed
% identically equal to zero, we use the notation $=$ and $\equiv$
% respectively.}:
\vskip 0.25cm
\begin{tabular}{cc}
{\bf I)} $\tQ_{P^4}\equiv 0$ and $\tQ_{P^2X^2} \equiv 0$, &\hskip 1cm {\bf III)} $\tQ_{P^4}\equiv 0$ and $\tQ_{P^2X^2} \not\equiv 0$,\\
{\bf II)} $\tQ_{P^4}\not\equiv 0$ and $\tQ_{P^2X^2} \equiv 0$, &\hskip 1cm {\bf IV)} $\tQ_{P^4}\not\equiv 0$ and $\tQ_{P^2X^2}\not\equiv 0$.
\end{tabular}
\vskip 0.25cm
Branch {\bf I} corresponds to the equivalence class of linearizable
differential equations discussed in the introduction.  For this class
of equations we have $\tvarrho_4=0$, and so the IC order equals $4$.

As for case {\bf IV}, we see from the recurrence relations \eqref{45
  recurrence relations} that the Maurer--Cartan forms
\eqref{unnormalized mc forms} can be normalized by setting 
$$
\tQ_{P^4},\tQ_{P^2X^2}\to1,\qquad \tQ_{P^5},\tQ_{P^4U},\tQ_{P^4X}\to0.
$$
This uses up all of the remaining $H$-freedom and produces a genuine
moving frame.  The algebra of absolute differential invariants is then
generated by the remaining invariants \eqref{456 invariants} of order
5 and 6.  The reduced ranks are
\[ \tvarrho_4 =2, \qquad \tvarrho_5 \geq 5 .\] Hence, the ``worst-case
scenario'', as far as the IC order is concerned, is
\[ (\tvarrho_5,\tvarrho_6,\tvarrho_7,\tvarrho_8,\tvarrho_9) =
(5,6,7,8,8);\] and as a consequence, the highest IC order achievable
is 9.  This bound will be attained if, post-normalization, the
remaining 5th order invariants are constant, and the signature
manifold is parametrized by three invariants of order $6,7,8$,
respectively, with the higher order invariants obtained by
differentiating the order 6 invariant.  We do not push the analysis
further as we will show that cases {\bf II} and {\bf III} contain
equations with invariant classification order equal to 10.  Since our
goal is to find the branch(es) with highest classification order, we
will focus on those branches.  Indeed, in the sequel we consider case
{\bf III} in detail, as branches {\bf II} and {\bf III} are dual to
each other.  This duality was first observed by Cartan in his study of
projective connections, \cite{C-1955}.  For completeness, the duality
among second-order ordinary differential equations is presented in
Appendix \ref{duality appendix}.

%%%%%
\subsection{Case III}
\label{sect:case3}
%%%%% 
From now on, we assume that $\tQ_{P^2X^2}\not\equiv 0$ and
$\tQ_{P^4}\equiv 0$.  With these assumptions, we will show that the
reduced rank sequence obeys
\[ \tvarrho_4 = 1,\qquad \tvarrho_5 =4,\qquad \tvarrho_6 \geq 5.\]
Hence, for this class of ODEs there exists a genuine moving frame
formulated in term of 6th order jets.  For class \textbf{III}
equations, the ``worst case scenario'' is the rank sequence
\[ (\tvarrho_4,\tvarrho_5,\tvarrho_6,\tvarrho_7,\tvarrho_8,\tvarrho_9,\tvarrho_{10}) =
(1,4,5,6,7,8,8) ;\]
which makes an IC order of 10 a possibility.

By Proposition \ref{prop:QP2X4-QP3X3}, below, the class \textbf{III}
has two branches, which we label \textbf{III.1}
and \textbf{III.2}.  For sub-case \textbf{III.1} we will show that
$\tvarrho_6=5$ implies $\tvarrho_7=5$, which means that the IC order
is 7.  The other possibility is that $\tvarrho_6\geq 6$, but this
means that the IC order is $\leq 9$.  Hence, sub-case \textbf{III.1}
can be ruled out.

Finally, for sub-case \textbf{III.2} we will show that there is
essentially one type of configuration of invariant values that gives the
rank sequence $(\tvarrho_4,\tvarrho_5,\tvarrho_6,\tvarrho_7,\tvarrho_8,\tvarrho_9,\tvarrho_{10}) =(1,4,5,6,7,8,8)$.  We will derive this
configuration, and in the subsequent section integrate the
corresponding structure equations.

Under the non-degeneracy assumption
$\tQ_{P^2X^2}\not\equiv 0$ it is possible to normalize
\begin{equation}\label{eq:QP2X2norm}
\tQ_{P^2X^2}\to1,\qquad \tQ_{P^3X^2},\tQ_{P^2UX^2},\tQ_{P^2X^3}\to0,
\end{equation}
which consequently normalizes the Maurer--Cartan forms $\tnu_U$,
$\tnu_U$, $\tnu_{UU}$, $\tnu_{UX}$ to certain linear combinations of
$\tomega^x,\tomega^u,\tomega^p,\tmu_X$.  Henceforth, to avoid
confusion, we use the ``check'' $\ckQ_{ijk}$ decoration to indicate
the invariants and one-forms obtained via additional normalization of
the universal invariants as per \eqref{eq:QP2X2norm}.  Furthermore,
the recurrence relation for $\tQ_{P^4}\equiv 0$, forces the following
fifth-order invariants:
\begin{equation}\label{QP4=0 5 order consequences}
\ckQ_{P^5}\equiv \ckQ_{P^4U} \equiv \ckQ_{P^4X}\equiv 0
\end{equation}
to be identically equal to zero.  Combining \eqref{eq:QP2X2norm} and \eqref{QP4=0 5 order consequences} we conclude that
all universal 5th order invariants can be normalized to a constant, and
that to normalize the remaining Maurer--Cartan form $\ckmu_X$ we must
consider universal invariants of order 6.  At order 6, the constraints
\eqref{QP4=0 5 order consequences} force the invariants
\begin{equation}\label{zero 6th order invariants}
\ckQ_{P^6}\equiv \ckQ_{P^5 U} \equiv \ckQ_{P^5 X} \equiv \ckQ_{P^4 U^2} \equiv \ckQ_{P^4 UX} \equiv \ckQ_{P^4 X^2} \equiv 0
\end{equation}
to be identically zero, and more generally,
\begin{equation}\label{zero invariants}
\ckQ_{P^{4+i}U^j X^k}\equiv 0,\qquad i,j,k \geq 0.
\end{equation}
Hence from \eqref{zero 6th order invariants} we conclude that the
remaining non-constant universal invariants of order 6 are
$$
\ckQ_{P^3 U X^2},\qquad \ckQ_{P^3 X^3},\qquad \ckQ_{P^2 U^2 X^2},\qquad
\ckQ_{P^2 U X^3},\qquad \ckQ_{P^2 X^4}.
$$

\begin{proposition}\label{prop:QP2X4-QP3X3}
  The invariants $\ckQ_{P^2X^4}$ and $\ckQ_{P^3X^3}$ cannot
  simultaneously be equal to zero.
\end{proposition}

\begin{proof}
Considering the recurrence relations for $\ckQ_{P^3X^3}$, $\ckQ_{P^2UX^3}$, $\ckQ_{P^2 X^4}$ we have
\begin{align*}
  d \ckQ_{P^3X^3} &= \frac{3}{4} \ckQ_{P^3UX^2} \,\ckomega^p +
  \bigg(\ckQ_{P^3 U X^3} - \frac{9}{4} \ckQ_{P^2 U^2 X^2} \bigg)
  \ckomega ^u + \bigg( \ckQ_{P^3 X^4} - \frac{9}{4} \ckQ_{P^2 U
    X^3}\bigg) \ckomega^x \\
  &\quad + 2 \ckQ_{P^3X^3}\, \ckmu_X,\\
  d \ckQ_{P^2 U X^2} &= \bigg( \ckQ_{P^3UX^3} - \ckQ_{P^2U^2X^2} -
  \frac{1}{5}\ckQ_{P^3X^3}^2 \bigg)\ckomega^p + \bigg(
  \ckQ_{P^2U^2X^3}
  - \ckQ_{P^3UX^2}\,\ckQ_{P^2X^4} \\
  &\quad - \frac{1}{5}\ckQ_{P^3X^3}\,\ckQ_{P^2UX^3}\bigg)\ckomega^u + \bigg(
  \ckQ_{P^2 UX^4} - \frac{5}{6} -
  \frac{6}{5}\ckQ_{P^3X^3}\,\ckQ_{P^2X^4}\bigg)\ckomega^x +
  \ckQ_{P^2UX^3}\,\ckmu_X,\\
  d \ckQ_{P^2 X^4} &= \ckQ_{P^3X^4}\,\ckomega^p +
  \ckQ_{P^2UX^4}\,\ckomega^u + \ckQ_{P^2X^5}\,\ckomega^x - 2
  \ckQ_{P^2X^4}\,\ckmu_X.
\end{align*}
Assuming $\ckQ_{P^2X^4}\equiv \ckQ_{P^3X^3}\equiv 0$, the syzygy
$$
D_X(\ckQ_{P^3X^3})+\frac{9}{4}\ckQ_{P^2UX^3} = \ckQ_{P^3 X^4} = D_P(\ckQ_{P^2X^4})
$$
forces $\ckQ_{P^2UX^3}\equiv 0$, which when combined with the syzygy
$$
D_X(\ckQ_{P^2UX^3})+\frac{5}{6} + \frac{6}{5} \ckQ_{P^3X^3} \ckQ_{P^2X^4} = \ckQ_{P^2UX^4} = D_U(\ckQ_{P^2X^4}),
$$
leads to the contradiction $0=5/6$.
\end{proof}

By virtue of Proposition \ref{prop:QP2X4-QP3X3}, two sub-cases must be
considered: \vskip 0.25cm
\begin{tabular}{cc}
{\bf III.1)}  $\ckQ_{P^3X^3} \not\equiv 0$, & \hskip 1cm
{\bf III.2)}  $\ckQ_{P^2X^4} \not\equiv 0$.
\end{tabular}
 
%%%%%
\subsubsection{Sub-case III.1}
%%%%%

Assuming $\ckQ_{P^3X^3} \not\equiv 0$, we can set 
$$
\ckQ_{P^3X^3}\to1
$$
and normalize the Maurer--Cartan form $\ckmu_X$.  After normalization,
the remaining sixth-order invariants
\begin{equation}\label{6 genuine invariants sub-case 1-case 3}
\ckQ_{P^3UX^2},\qquad \ckQ_{P^2 U^2 X^2},\qquad \ckQ_{P^2 U X^3},\qquad \ckQ_{P^2 X^4},
\end{equation}
are genuine invariants in the sense that they do not depend on
pseudo-group parameters.  In an attempt to minimize the rank, we
assume that the functions \eqref{6 genuine invariants sub-case 1-case
  3} are constant:
\begin{equation}\label{constant assumption-case 3}
\ckQ_{P^3 U X^2} \equiv C_1,\qquad \ckQ_{P^2 U^2 X^2} \equiv C_2,\qquad \ckQ_{P^2U X^3}\equiv C_3,\qquad \ckQ_{P^2 X^4}\equiv C_4.
\end{equation}
Combining \eqref{zero 6th order invariants} with \eqref{constant
  assumption-case 3}, it follows from a careful analysis of the
recurrence relations that all seventh-order invariants are constant
which in turns forces all higher-order invariants to also be constant.
On the other hand, if the invariants \eqref{6 genuine invariants
  sub-case 1-case 3} are not constant, the IC order is $\leq 9$.

%%%%%
\subsubsection{Sub-case III.2}
%%%%% 
 
We now assume that $\ckQ_{P^2X^4} \not\equiv 0$, and set
\begin{equation}
  \label{eq:QP2X4norm}
  \ckQ_{P^2X^4}\to1
\end{equation}
to normalize $\ckmu_X$ and obtain a genuine moving frame.  From now
on, for the sake of notational convenience, we omit writing the check
decoration and simply use $Q_{ijk}$ to denote the absolute
differential invariants obtained by normalizing the universal
invariants using \eqref{eq:QP2X2norm} and \eqref{eq:QP2X4norm}.
Likewise, the invariant coframe on $N^3$ will be written simply as
$\omega^x, \omega^u, \omega^p$ and the dual derivative operators as
$D_X, D_U, D_P$.

At order 6, we are left with the absolute differential invariants
\begin{equation}\label{remaining order 6 invariants}
Q_{P^3UX^2},\qquad Q_{P^3X^3},\qquad Q_{P^2U^2X^2},\qquad Q_{P^2UX^3}.
\end{equation}
Once more, in an attempt to minimize the rank, we assume that the
invariants \eqref{remaining order 6 invariants} are constant:
\begin{equation}\label{constant order 6 invariants}
Q_{P^3UX^2}\equiv C_1,\qquad Q_{P^3X^3} \equiv C_2,\qquad Q_{P^2U^2X^2}\equiv C_3,\qquad Q_{P^2UX^3}\equiv C_4.
\end{equation}
The seventh-order invariant $Q_{P^2X^5}$ plays an important role in
the following considerations.  To single out this invariant, let us set
\[
I_7=Q_{P^2X^5}.
\]
Considering the recurrence relations of the invariants \eqref{constant
  order 6 invariants} we obtain a collection of constraints on the
seventh-order invariants:
\begin{align}
  0=dC_1 &= \frac{5 C_1Q_{P^3X^4}}{2} \,\omega^p + \bigg(
  Q_{P^3U^2X^2} - C_1C_2 + \frac{5C_1
    Q_{P^2UX^4}}{2}\bigg)\omega^u\nonumber\\ 
  &\quad + \bigg( Q_{P^3 UX^3} - C_2^2 - 2C_3 + \frac{5
    I_7C_1}{2}\bigg)\omega^x,\nonumber\\
  0=dC_2 &= \bigg(C_2Q_{P^3X^4}+\frac{3C_1}{4}\bigg)\omega^p+\bigg(Q_{P^3UX^3} + C_2Q_{P^2UX^4}-\frac{9C_3}{4}\bigg)\omega^u\nonumber\\
  &\quad +\bigg(Q_{P^3X^4}+I_7C_2-\frac{9C_4}{4}\bigg)\omega^x,\label{case
    3 order 6 recurrence relations}\\
  0=dC_3 &= \bigg( Q_{P^3U^2X^2} +2 C_3 Q_{P^3X^4}-\frac{2
    C_1C_2}{5} \bigg)\omega^p + \bigg( Q_{P^2U^3X^2} + 2
  C_3Q_{P^2UX^4}\nonumber\\
  &\quad -\frac{12 C_1C_4}{5} \bigg) \omega^u  + \bigg( Q_{P^2U^2X^3}
  + 2I_7C_3 -2C_2C_4 - \frac{2C_1}{5}\bigg)\omega^x,\nonumber\\ 
  0=dC_4 &= \bigg( Q_{P^3UX^3}  - C_3 + \frac{C_4Q_{P^3X^4}}{2}
  - \frac{C_2^2}{5}\bigg)\omega^p + \bigg( Q_{P^2U^2X^3} - C_1 +
  \frac{C_4Q_{P^2UX^4}}{2} \nonumber\\ 
  &\quad  - \frac{C_2C_4}{5} \bigg)\omega^u + \bigg( Q_{P^2UX^4} +
  \frac{I_7 C_4}{2} - \frac{6C_2}{5} - \frac{5}{6}\bigg)
  \omega^x.\nonumber
\end{align}
If the invariant $I_7$ is constant, \eqref{zero invariants} and
\eqref{case 3 order 6 recurrence relations} imply that all
seventh-order invariants are constant. Similarly, all higher order
invariants are constant.  On the other hand, when $I_7$ is a
non-constant invariant, the constraints \eqref{case 3 order 6
  recurrence relations} yield
$$
C_1=C_2=C_3=C_4=0.
$$
Which in turn, implies, together with \eqref{zero invariants}, that all seventh-order invariants are identically equal to zero except for
$$
Q_{P^2X^5}=I_7\qquad\text{and}\qquad Q_{P^2 U X^4}=\frac{5}{6}.
$$
Taking the exterior derivative of $I_7$ we obtain
\begin{equation}\label{dI}
dI_7 = I_8 \,\omega^x - \frac{5}{4} I_7 \,\omega^u + \frac{5}{6} \,\omega^p,
\end{equation}
where $I_8=D_X(I_7)$ is the only new (functionally independent)
invariant of order 8.  Then, differentiating $I_8$ with respect to $D_X$
we find the only new invariant of order $9$:
$$
I_9=D_X(I_8)=D_X^2(I_7).
$$
Generically, the invariants $I_7, I_8, I_9$, are functionally
independent, and the structure of the invariant signature manifold is
completely determined by the functional relation
\[
I_{10}=D_X(I_9)=\phi(I_7,I_8,I_9).
\]
Modulo duality, all branches of the equivalence problem have now been
considered, and we can safely conclude that $10$ is an upper bound on
the IC order.  

\section{The maximal IC order class}
\subsection{Abstract existence}
To terminate the proof of Theorem \ref{thm:main}, we
must show that there exists a class of differential equations
satisfying the invariant constraints imposed in sub-case {\bf III.2}.
To do so, we need the structure equations of the invariant one-forms
$\omega^x$, $\omega^u$, $\omega^p$.

These equations are obtained symbolically by substituting the
Maurer--Cartan form normalizations
\begin{gather*}
  \mu=-\omega^x,\qquad \nu=-\omega^u,\qquad \nu_X=-\omega^p,\\
  -2\mu_X = \nu_U = -I_7 \,\omega^x -\frac{5}{6}\,\omega^u,\qquad
  \mu_U = \nu_{XX}= 0,\qquad \nu_{UX} = \frac{1}{5}\omega^x,
\end{gather*}
obtained by solving the recurrence relations for the phantom invariants,  into the structure equations \eqref{eq:domegaxup}.  The result is
\begin{equation}\label{structure equations-case3}
\begin{aligned}
  d\omega^x &= \frac{5}{12} \,\omega^u \wedge \,\omega^x,\\
  d\omega^u &= \omega^x \wedge \,\omega^p + I_7 \,\omega^u \wedge \,\omega^x,\\
  d\omega^p &= \frac{3}{2} I_7 \,\omega^p \wedge \,\omega^x +
  \frac{5}{4} \,\omega^p \wedge \,\omega^u + \frac{1}{5} \,\omega^x
  \wedge \,\omega^u.
\end{aligned}
\end{equation}

If the invariant $I_7$ were constant, we could apply Cartan's
Integration Theorem, \cite{BCGGG-1991}, to conclude the existence of
differential equations solving the integration problem.  But since
this is not the case we must use the following generalization of
Cartan's integration theorem \cite{B-2011}.
\begin{theorem}\label{Cartan theorem}
  Let $\,\omega^1,\ldots,\,\omega^\ell$ be a coframe with structure
  equations
  $$
  d\omega^i= \sum_{1\leq j < k \leq \ell} C^i_{jk}(I^a)\,\,\omega^j \wedge
  \,\omega^k,
  $$
  such that the structure coefficients are functions of $I^a, \;
  a=1,\ldots, s$ and
  $$
  dI^a= \sum_{i=1}^\ell \bigg(F^a_i(I^b) + \sum_{\alpha=1}^r A^a_{i
    \alpha}(I^b)J^\alpha\bigg)\,\omega^i,\qquad a,b=1,\ldots,s.
  $$
  Assuming 
  \begin{itemize}
  \item the identity $d^2=0$ holds,
  \item the functions $C^i_{jk}$, $F^a_i$, $A^a_{i \alpha}$ are real analytic,
  \item the tableau $A(I^b)=(A^a_{i\alpha}(I^b))$ has rank $r$ and is
    involutive with Cartan characters $s_1 \geq s_2 \geq \cdots \geq
    s_q > s_{q+1} =0$ for all values of $(I^b)$,
  \end{itemize}
  modulo a diffeomorphism, the general real-solution exists and
  depends on $s_q$ functions of $q$ variables.  Moreover, $(I^a)$ and
  $(J^\alpha)$ can be arbitrarily specified at a point.
\end{theorem}

Applying Theorem \ref{Cartan theorem} to \eqref{dI} and
\eqref{structure equations-case3} we conclude that there exists a
family of second-order ordinary differential equations depending on
one function of one variable that solves the integration problem.
Hence, by duality we conclude the there are two families of equations,
both depending on one function of one variable, that achieve the
maximal classification order of Theorem \ref{thm:main}.

%By duality, we in fact know that there are two classes of differential equations, both depending on one function of one variable whose invariant classification order is 9.

%%%%%
\subsection{Explicit integration}
\label{sect:integration}
%%%%%

In this section we integrate the structure equations \eqref{structure
  equations-case3} to obtain an explicit representation of one of the
two families the differential equations satisfying the maximal
classification order of Theorem \ref{thm:main}.

\begin{proposition}
  Let $(\alpha, \beta, \gamma)$ be a local coordinate system on
  $\mathbb{R}^3\setminus \{\alpha =0\}$, then the one-forms
  \begin{equation}\label{integrated one-forms}
    \omega^x = \alpha\, d\gamma,\qquad 
    \omega^u = \frac{12}{5} \frac{d\alpha}{\alpha} - \beta\, d\gamma,\qquad
    \omega^p = \frac{d\beta}{\alpha} + \frac{12 I_7}{5}
    \frac{d\alpha}{\alpha} + \Gamma(\alpha, \beta, \gamma)\, d\gamma, 
  \end{equation}
  where $\Gamma(\alpha,\beta, \gamma)$ is a solution of the of linear
  partial differential equations
  \begin{equation}\label{PDE for Gamma}
    \Gamma_\alpha + \frac{3}{\alpha}\Gamma = \frac{12 I_{7,\gamma}}{5
      \alpha} + \frac{18 I_7^2}{5} - \frac{3 I_7
      \beta}{\alpha}-\frac{12}{25},\qquad
    \Gamma_\beta=\frac{3I_7}{2}-\frac{5 \beta}{4\alpha}, 
  \end{equation}
  satisfy the structure equations \eqref{structure equations-case3}.  
\end{proposition}

Substituting \eqref{integrated one-forms} into \eqref{dI} we obtain
the differential equations
\[ I_{7,\beta} = \frac{5}{6 \alpha},\qquad \alpha I_{7,\alpha} = -
I_7,\qquad I_8=\frac{I_{7,\gamma}}{\alpha} - \frac{5 \beta I_7}{4
  \alpha} - \frac{5 \Gamma}{6\alpha}.
\]
Integrating the first two equations, we deduce that
\begin{equation}\label{invariant I}
I_7 = \frac{5 \beta}{6 \alpha} + \frac{h(\gamma)}{\alpha},
\end{equation}
where $h(\gamma)$ is an arbitrary (analytic) function, while the third
equation defines $I_8$ in terms of $I_7$ and $\Gamma$. The coframe
\eqref{integrated one-forms} is not uniquely defined.  The degree of
freedom is given by the infinite-dimensional Lie pseudo-group
\begin{equation}\label{coframe liberty}
  \begin{gathered}
    \overline{\gamma} = \sigma(\gamma),\qquad
    \overline{\alpha}=\frac{\alpha}{\sigma^\prime},\qquad
    \overline{\beta}=\frac{\beta}{\sigma^\prime} - \frac{12
      \sigma^{\prime\prime}}{5 (\sigma^\prime)^2},\\ 
    \overline{\Gamma}=\frac{\Gamma}{\sigma^\prime} + \frac{12
      \overline{I_7} \,\sigma^{\prime\prime}}{5 (\sigma^\prime)^2} +
    \frac{\beta \sigma^{\prime\prime}}{\alpha(\sigma^\prime)^2} +
    \frac{12\sigma^{\prime\prime\prime}}{5 \alpha (\sigma^\prime)^2}
    - \frac{24 (\sigma^{\prime\prime})^2}{5 \alpha
      (\sigma^\prime)^3},
\end{gathered}
\end{equation} 
where $\overline{\gamma}=\sigma(\gamma)$ is a local diffeomorphism of
the real line.  Under the pseudo-group action \eqref{coframe liberty},
the invariant \eqref{invariant I} transforms according to
\[
\overline{I_7}=\frac{5\beta}{6 \alpha} - \frac{2
  \sigma^{\prime\prime}}{\alpha \sigma^\prime} + \frac{\sigma^\prime
  \overline{h}}{\alpha},\qquad \text{with}\qquad \overline{h}=h \comp
\sigma.
\]
By choosing $\sigma(\gamma)$ such that 
\[ 2 \sigma^{\prime\prime}=(\sigma^\prime)^2 \overline{h},\] we can
assume $h(\gamma) = 0$ in \eqref{invariant I}.  Doing so and solving
the differential equations \eqref{PDE for Gamma} we find that
\[ \Gamma(\alpha, \beta, \gamma)=- \frac{3 \alpha}{25} +
\frac{g(\gamma)}{\alpha^3}, \]
where $g(\gamma)$ is an arbitrary (analytic) function which cannot be
removed by some change of variables.

\begin{proposition}
The one-forms
\begin{equation}\label{coframe coordinate expressions}
  \omega^x = \alpha\, d\gamma,\qquad 
  \omega^u = \frac{12}{5} \frac{d\alpha}{\alpha} - \beta\, d\gamma,\qquad
  \omega^p = \frac{d\beta}{\alpha} + \frac{2 \beta}{\alpha^2} d\alpha +
  \bigg(\frac{g(\gamma)}{\alpha^3}-\frac{3\alpha}{25}\bigg) d\gamma 
\end{equation}
satisfy the structure equations \eqref{structure equations-case3} with
\begin{equation}\label{I - h=0}
I_7(\alpha,\beta,\gamma)=\frac{5\beta}{6\alpha}.
\end{equation}
\end{proposition} 
 
\begin{proposition}
Let $\omega^x$, $\omega^u$, $\omega^p$ be given by \eqref{coframe coordinate expressions}.  Then there exist functions $a_1$, $a_2$, $a_3$, $a_4$ and a change of variables
\begin{equation}\label{change of variables}
\alpha=\alpha(x,u,p),\qquad \beta=\beta(x,u,p),\qquad \gamma=\gamma(x,u,p)
\end{equation}
such that \eqref{G structure} holds.
\end{proposition}

\begin{proof}
The functions $a_1$, $a_2$, $a_3$, $a_4$ and the change of variables \eqref{change of variables} are not unique.  Assuming $u>0$, we can choose
$$
a_1=\frac{6}{5\alpha^2},\qquad a_2=a_3=0,\qquad a_4= \alpha, 
$$
and
$$
\alpha=\sqrt{u},\qquad \beta=\frac{6 p}{5 u},\qquad \gamma = x.
$$
Then
\begin{equation}\label{final integrated coframe}
\omega^x = u^{1/2} dx,\qquad \omega^u = \frac{6}{5 u}[du - p\, dx],\qquad \omega^p = \frac{6}{5 u^{3/2}} \bigg[dp - \bigg( \frac{u^2}{10}-\frac{5 g(x)}{6} \bigg) dx\bigg].
\end{equation}
\end{proof}

We deduce from $\omega^p$ in \eqref{final integrated coframe} that the
second-order ordinary differential equations
\begin{equation}\label{representative ODE}
  u_{xx}=\frac{u^2}{10} - \frac{5 g(x)}{6},
\end{equation}
is one of the 2 families of differential equations with maximal
classification order, provided $I_7, I_8, I_9$ are functionally
independent.  A straightforward scaling transformation gives the form
shown in \eqref{eq:ord4form}.  We note that the coordinate expressions
for the dual equations to \eqref{representative ODE} are very
difficult to obtain.
% that when $u$ is scaled by a factor $\lambda = 60$ and $f(x)=-72 x$ in
% \eqref{representative ODE} we obtain the Painlev\'e I equation.

Finally, for the differential equations \eqref{representative ODE},
the coordinate expressions of the invariant \eqref{I - h=0} and the
derivatives dual to $\omega^x,\omega^u,\omega^p$ are
\begin{equation}\label{explicit I}
  I_7 = \frac{p}{u^{3/2}},
\end{equation}
and 
\begin{gather*}
  D_X = \frac{1}{u^{1/2}}\bigg[ \frac{\partial}{\partial x} + p\,
  \frac{\partial}{\partial u} + \bigg( \frac{u^2}{10} - \frac{5
    g(x)}{6}\bigg)\frac{\partial}{\partial p} \bigg],\quad D_U =
  \frac{5 u}{6} \frac{\partial}{\partial u},\quad D_P = \frac{5
    u^{3/2}}{6} \frac{\partial}{\partial p},
\end{gather*}
respectively.  Hence, differentiating \eqref{explicit I} twice with respect to $D_X$ we deduce that the invariant signature manifold can be parametrized by the invariants
\[
I_7=\frac{p}{u^{3/2}},\qquad I_8=\frac{g(x)}{u^2},\qquad
I_9=\frac{(g^\prime(x))^4}{g^5(x)} = 256 \left(D_x( g(x)^{-1/4})\right)^4.
\]
The above invariants are independent if and only if $I_9$ is
non-constant, which gives the inequality in \eqref{eq:ord4form}.  If
this inequality holds, then the invariant classification of an
equation in this class is completely determined by the functional
relationship between $I_9$ and the invariant
\[
I_{10}=\frac{g(x) g^{\prime\prime}(x)}{(g^\prime(x))^2}.
\]

\begin{remark}
Now one can understand the necessity of \eqref{eq:ord4form} in Theorem \ref{thm:main}.  Indeed, when $I_9$ is constant, which is equivalent to the requirement that 
$$
D_x^2[g(x)^{-1/4}]=0,
$$ 
the IC order is 9 and the coframe \eqref{final integrated coframe} admits a 1-dimensional symmetry group, \cite[Theorem 8.22]{O-1995}.
\end{remark}

%%%%%
\begin{appendix}
%%%%%

\section{Cartan's duality for second-order ODEs}
\label{duality appendix}

In his study of projective connections, \cite{C-1955}, Cartan mentions
that there exists a notion of duality among second-order ordinary
differential equations.  Modern accounts can be found in
\cite{CS-2005,NS-2003}, but for completeness, we summarize the
construction below.

For second-order ordinary differential equations, a solution to an
initial-value problem may be represented by a two-parameter family
\begin{equation}\label{Phi}
\Phi(x,u,\overline{x},\overline{u})=0,
\end{equation}
where the parameters $\overline{x}$, $\overline{u}$ correspond to the initial values
\begin{equation*}
\overline{x}=u(0),\qquad \overline{u}=u_x(0).
\end{equation*}
Differentiating \eqref{Phi} twice with respect to $x$ we obtain the equations
\begin{equation}\label{equations defining q}
\Phi=\Phi_x + u_x \Phi_u = \Phi_{xx} + 2 u_x \Phi_{xu} + u_x^2 \Phi_{uu} + u_{xx} \Phi_u = 0, 
\end{equation}
which leads to the second-order differential equation 
\begin{equation}\label{second order ODE}
u_{xx}= q(x,u,u_x)
\end{equation}
when the parameters $\overline{x}$ and $\overline{u}$ are eliminated from \eqref{equations defining q}.  The dual equation to \eqref{second order ODE} is obtained from \eqref{Phi} by interchanging the roles of $(x,u)$ and $(\overline{x},\overline{u})$.  In other words, let $(x,u)$ be parameters and $\overline{u}=\overline{u}(\overline{x})$ a function of the independent variable $\overline{x}$.  Then, differentiating \eqref{Phi} with respect to $\overline{x}$ twice and getting rid of $(x,u)$ from the  equations obtained yields the \emph{dual equation}
\begin{equation}\label{dual equation}
\ou_{\ox\ox}=\overline{q}(\ox,\ou,\ou_{\ox}).
\end{equation}

Let us now consider the consequences of the contact transformation that sends \eqref{second order ODE} to \eqref{dual equation} on the point equivalence problem.   For the equations
$$
\Phi(x,u,\ox,\ou)=0,\qquad \Phi_x(x,u,\ox,\ou)+ p\, \Phi_u(x,u,\ox,\ou)=0,
$$
to determine $\ox$, $\ou$, $\op$ in terms of $x$, $u$, $p$ we must impose
$$
0\neq \Delta = \text{det}\begin{pmatrix}
0 & \Phi_{\ox} & \Phi_{\ou} \\
\Phi_x & \Phi_{x\ox} & \Phi_{x\ou} \\
\Phi_u & \Phi_{u\ox} & \Phi_{u\ou}
\end{pmatrix}
=-\Phi_u \Phi_{\ou}(\Phi_{x\ox} + p\, \Phi_{u\ox} + \op\, \Phi_{x\ou} + p\, \op\, \Phi_{u \ou}),
$$
which in particular requires
$$
\Phi_u\neq 0,\qquad \Phi_{\ou}\neq 0.
$$
Now, let
$$
\theta=du-p\,dx,\qquad \overline{\theta}=d\ou - \op\, d\ox,\qquad \theta_1 = dp - q\, dx,\qquad \overline{\theta}_1 = d\op - \overline{q}\, d\ox.
$$
Taking the exterior derivative of \eqref{Phi} we deduce that
$$
\theta = -\frac{\Phi_{\ou}}{\Phi_u} \,\overline{\theta}.
$$
On the other hand, the exterior derivative of $\Phi_x + p\, \Phi_u=0$ yields
$$
\theta_1 = \frac{\Delta}{\Phi_{\ou} \Phi_u^2} \,d\ox \quad \text{mod } \overline{\theta},
$$
while the exterior derivative of $\Phi_{\ox} + \op\, \Phi_{\ou} = 0$ gives
$$
dx = \frac{\Phi_{\ou}^2 \Phi_u}{\Delta} \,\overline{\theta}_1\quad \text{mod } \overline{\theta}.
$$
In matrix form,
$$
\begin{pmatrix}
\theta \\ dx \\ \theta_1
\end{pmatrix} 
=
\begin{pmatrix}
-a_1 & 0 & 0 \\
a_2 & a_1/a_4 & 0 \\
a_3 & 0 & a_4
\end{pmatrix}
\begin{pmatrix}
\overline{\theta} \\ \overline{\theta}_1 \\ d\ox
\end{pmatrix},
\qquad \text{where}\qquad a_1 = \frac{\Phi_{\ou}}{\Phi_u},\qquad a_4 = \frac{\Delta}{\Phi_{\ou} \Phi_u^2}.
$$
We note that, up to a sign in the first entry, the $3\times 3$ matrix
made of the functions $a_1$, $a_2$, $a_3$, $a_4$ is an element of the
structure group \eqref{G structure}.  Hence, under the contact
transformation $(x, u, p) \leftrightarrow (\ox, \ou, \op)$, the lifted
coframe \eqref{eq:omxudef} and \eqref{eq:ompdef} undergoes the
transformation
\begin{equation}\label{coframe transformation}
\omega^x \leftrightarrow \omega^{\op},\qquad \omega^p \leftrightarrow \omega^{\ox},\qquad \omega^u \leftrightarrow - \omega^{\ou}.
\end{equation}

To understand the duality between cases {\bf II} and {\bf III} of the
equivalence problem in Section \ref{sect:branching} we consider the
structure equations \eqref{eq:uredseq} obtained once the
normalizations \eqref{eq:liftedunorms} are done.
% %
% \begin{gather*}
% \mu = - \omega^x,\quad 
% \nu = -\omega^u,\quad 
% \nu_X = -\omega^p,\quad 
% \nu_{XX}= \nu_{UXX}= 0,\quad 
% \mu_{XX}= 2 \nu_{UX},\\
% %
% \mu_{UX}= \frac{\nu_{UU}}{2},\quad
% \mu_{UU}= \frac{Q_{P^3}}{6} \omega^p,\quad 
% \nu_{UUX}=\frac{Q_{P^2X^2}}{6} \omega^x,\quad
%  \nu_{UUU}= \frac{Q_{P^4X}}{3} \omega^p + \frac{Q_{P^3X^2}}{3} \omega^x,
% \end{gather*}
% %
% into the structure equations \eqref{mc structure equations}.  The result is
% %
% \begin{equation}\label{intermediate structure equations}
% \begin{aligned}
% &d\omega^x = \mu_X \wedge \omega^x + \mu_U \wedge \omega^u,\\
% %
% &d\omega^u =  \nu_U \wedge \omega^u + \omega^x \wedge \omega^p,\\
% %
% &d\omega^p = \nu_{UX} \wedge \omega^u + (\nu_U - \mu_X) \wedge \omega^p,\\
% %
% &d\mu_X = -2\,\nu_{UX} \wedge \omega^x - \frac{\nu_{UU}}{2} \wedge \omega^u - \mu_U \wedge \omega^p,\\
% %
% &d\mu_U = -\frac{\nu_{UU}}{2}\wedge \omega^x + (\mu_X - \nu_U)\wedge \mu_U + \frac{Q_{P^4}}{6} \omega^u \wedge \omega^p,\\
% %
% &d\nu_U = -\nu_{UX} \wedge \omega^x + \mu_U \wedge \omega^p - \nu_{UU}\wedge \omega^u,\\
% %
% &d\nu_{UU} = 2 \nu_{UX} \wedge \mu_U + \nu_{UU}\wedge \nu_U + \frac{Q_{P^4X}}{3} \omega^u\wedge \omega^p + \frac{Q_{P^3X^2}}{3}\omega^u \wedge \omega^x,\\
% %
% &d\nu_{UX} = \nu_{UX} \wedge \mu_X - \frac{1}{2} \nu_{UU} \wedge \omega^p + \frac{Q_{P^2X^2}}{6} \omega^u \wedge \omega^x.
% \end{aligned}
% \end{equation}
%
Focusing on the structure equations 
$$
d\tmu_U=\cdots + \frac{1}{6}\tQ_{P^4}\,\tomega^u \wedge \tomega^p,\qquad
d\tnu_{UX}=\cdots + \frac{1}{6}\tQ_{P^2X^2}\, \tomega^u \wedge \tomega^x,
$$ 
we observe that under the coframe transformation \eqref{coframe
  transformation} the role of the universal invariants $\tQ_{P^4}$,
$\tQ_{P^2X^2}$ is interchanged, which is exactly what happens when
switching between cases {\bf II} and {\bf III}.

\section{Proofs of Propositions \ref{prop:norm} and \ref{prop:Hstab}}
\label{apdx:norm}
We start by proving Proposition \ref{prop:norm}.  Let us first
consider local transversality.  First, we observe that $dX, dU,
dP, dQ_{ijk}, \mu_{ij}, \nu_{ij}$ and $\omega^x, \omega^u,
\omega^p,\vartheta_{ijk},\mu_{ij}, \nu_{ij}$ are two choices of
right-invariant coframes on $\E\ii$.  Then, we note that the
left-action of $\Diff(\R^2)$ in \eqref{eq:phindef} generates the
source distribution of $\bs\colon\E\ii\to\J\ii$ given by
 \[ \ker \{ dx,du,dp,\theta_{ijk} \} = \ker \{
 \omega^x,\omega^u,\omega^p, \vartheta_{ijk} \},\] while the
 right-action generates the target distribution of $\bt\colon\E\ii\to
 \J\ii$, given by $ \ker \{ dX, dU, dP, dQ_{ijk} \}$.  Hence, our goal
 is to show that the tangent space to $\tE$, defined by equations
 \eqref{eq:liftedunorms}, is transverse to the source distribution on
 $\E\ii$.  

According to \eqref{infinitesimal
  generator},
\begin{align*}
  \phi(x,u,p,q) &= \eta_{xx} + q(\eta_u-2\xi_x) + p(2
  \eta_{xu}-\xi_{xx})-3pq\xi_u + p^2(\eta_{uu}-2\xi_{xu})-p^3
  \xi_{uu}.
\end{align*}
Writing
\begin{align*}
  &\xi_{ij} = \xi_{x^iu^j},& &\hskip -2cm\xi\n = \{ \xi_{ij}\colon 0\leq i+j\leq n \},\\
  &\eta_{ij} = \eta_{x^i u^j},& &\hskip -2cm\eta\n = \{ \eta_{ij} \colon 0 \leq
  i+j \leq n \},
\end{align*}
the prolongation formula \eqref{prolongation formula} yields, when $p=0$,
\begin{align*}
  \phi^{000} &\equiv \eta_{xx} &&\mod p, \xi^{(1)}, \eta^{(1)};\\
  \phi^{001} &\equiv  2\eta_{xu}-\xi_{xx}  &&\mod p, \xi^{(1)}, \eta^{(1)};\\
  \phi^{002} &\equiv 2\eta_{uu}-4\xi_{xu} &&\mod p, \xi^{(1)},
  \eta^{(1)};\\
  3\phi^{0,j-2,2} -2 \phi^{1,j-3,3} &\equiv 6 \eta_{0j} &&\mod
  \xi^{(j-1)},\eta^{(j-1)} ,\; j\geq 3;\\
  4\phi^{0,j-1,1} - \phi^{1,j-2,2} &\equiv 6 \eta_{1j} &&\mod
  \xi^{(j)},\eta^{(j)} ,\; j\geq 2;\\
  \phi^{i-2,j0} &\equiv \eta_{ij} &&\mod p,\xi^{(i+j-1)},
  \eta^{(i+j-1)},\; i\geq 2, j\geq 0;  \\
  \phi^{0,j-2,3} &\equiv - 6 \xi_{0j} &&\mod
  \xi^{(j-1)},\eta^{(j-1)} ,\; j\geq 2;\\
  \phi^{1,j-2,3} &\equiv -6 \xi_{1j} &&\mod
  \xi^{(j)},\eta^{(j)},\; j\geq 2 ;\\
  \phi^{0,j,1} - \phi^{1,j-1,2} &\equiv 3 \xi_{2j} &&\mod
  \xi^{(j+1)},\eta^{(j+1)},\; j\geq 1;\\
  \phi^{i-2,j,1} - 2 \phi^{i-3,j+1,0} &\equiv \xi_{ij} &&\mod
  \xi^{(i+j-1)},\eta^{(i+j-1)} ,\;  i\geq 3,j\geq 2.
\end{align*}
% \begin{align*}
%   \phi^{000} &\equiv \eta_{xx} &&\mod p, \xi^{(1)}, \eta^{(1)};\\
%   \phi^{001} &\equiv  2\eta_{xu}-\xi_{xx}  &&\mod p, \xi^{(1)}, \eta^{(1)};\\
%   \phi^{002} &\equiv 2\eta_{uu}-4\xi_{xu} &&\mod p, \xi^{(1)},
%   \eta^{(1)};\\
%   3\phi^{0,n-2,2} -2 \phi^{1,n-3,3} &\equiv 6 \eta_{0,n} &&\mod
%   p,\xi^{(n-1)},\eta^{(n-1)} ,\quad n\geq 3;\\
%   4\phi^{0,n-2,1} - \phi^{1,n-3,2} &\equiv 6 \eta_{1,n-1} &&\mod
%   p,\xi^{(n-1)},\eta^{(n-1)} ,\quad n\geq 3;\\
%   \phi^{ij0} &\equiv \eta_{i+2,j} &&\mod p,\xi^{(n-1)},
%   \eta^{(n-1)},\quad n=i+j+2\geq 3;  \\
%   \phi^{0,n-2,3} &\equiv - 6 \xi_{0,n} &&\mod
%   p,\xi^{(n-1)},\eta^{(n-1)} ,\quad n\geq 2;\\
%   \phi^{1,n-3,3} &\equiv -6 \xi_{1,n-1} &&\mod
%   p,\xi^{(n-1)},\eta^{(n-1)},\quad n\geq 3 ;\\
%   \phi^{0,n-2,1} - \phi^{1,n-3,2} &\equiv 3 \xi_{2,n-2} &&\mod
%   p,\xi^{(n-1)},\eta^{(n-1)},\quad n\geq 3;\\
%   \phi^{i+1,j-1,1} - 2 \phi^{i,j,0} &\equiv \xi_{3+i,j-1} &&\mod
%   p,\xi^{(n-1)},\eta^{(n-1)} ,\quad j\geq 1,\; n=i+j+2\geq 3.
% \end{align*}
By the universal recurrence formulas \eqref{eq:unirec}, when $P=0$ we have, modulo $
\omega^x, \omega^u, \omega^p, \vartheta_{ijk}$,
\begin{align}
  \nonumber
  d X & \equiv \mu;\\ \nonumber
  dU & \equiv \nu ;\\ \nonumber
  dP & \equiv \nu_X ;\\ \nonumber 
  dQ_{000} &\equiv \nu_{XX} &&\mod \mu^{(1)}, \nu^{(1)};\\ \nonumber
  dQ_{001} &\equiv 2\nu_{XU}-\mu_{XX} &&\mod \mu^{(1)}, \nu^{(1)};\\ \nonumber
  dQ_{002} &\equiv 2\eta_{UU}-4\mu_{XU} &&\mod \mu^{(1)},
  \nu^{(1)};\\ \nonumber
  3dQ_{0,j-2,2} -2 dQ_{1,j-3,3} &\equiv 6 \nu_{0j} &&\mod
  \mu^{(j-1)},\nu^{(j-1)} ,\; j\geq 3;\\ 
  \label{eq:nuxx}
  4dQ_{0,j-1,1} - dQ_{1,j-2,2} &\equiv 6 \nu_{1j} &&\mod
  \mu^{(j)},\nu^{(j)} ,\; j\geq 2;\\ \nonumber 
  dQ_{i-2,j0} &\equiv \nu_{ij} &&\mod \mu^{(i+j-1)},
  \nu^{(i+j-1)},\; i\geq 2,j\geq 0;  \\ \nonumber
  dQ_{0,j-2,3} &\equiv - 6 \mu_{0j} &&\mod
  \mu^{(j-1)},\nu^{(j-1)} ,\; j\geq 2;\\ \nonumber
  dQ_{1,j-2,3} &\equiv -6 \mu_{1j} &&\mod
  \mu^{(j)},\nu^{(j)},\; j\geq 2 ;\\ \nonumber
  dQ_{0,j,1} - dQ_{1,j-1,2} &\equiv 3 \mu_{2j} &&\mod
  \mu^{(j+1)},\nu^{(j+1)},\; j\geq 1;\\ \nonumber 
  dQ_{i-2,j,1} - 2 dQ_{i-3,j+1,0} &\equiv \mu_{ij} &&\mod
  \mu^{(i+j+1)},\nu^{(i+j+1)} ,\; i\geq 3,\;j\geq 2.
\end{align}
It follows that $\omega^x, \omega^u, \omega^p, \nu_U, \mu_X, \mu_U,
\nu_{XU} , \nu_{UU},\vartheta_{ijk}$ form a basis of $T^*\tE$.  As a
consequence, since, $\omega^x, \omega^u,\omega^p,\vartheta_{ijk}$ are
linearly independent on $\tE$, the latter is transverse to the source
distribution.

Now we prove surjectivity.  Let $J = \{ (i,j,k) \colon k<2 \text{ or }
(i<2 \text{ and } k<4) \}$ be the indicated set of indices.  We have
to show the consistency of the following equations:
\[ \hP(p,X_x,X_u,U_x,U_u)=0\] which is equivalent to
\[ p\, U_u + U_x = 0;\] and
\[ \hQ_{ijk}(p,q^{(i+j+k)}, X_0^{(i+j+k+2)}, U_0^{(i+j+k+2)})=0,\quad
(i,j,k) \in J. \] Since $\G$ is transitive on $N^3 = \J^1(\R,\R)$, no
generality is lost if we set $p=0$ above; that is, it suffices to
demonstrate the consistency of the equations
\[ 0 = \hQ_{ijk}(0,q^{(i+j+k)}, X_0^{(i+j+k+2)}, U_0^{(i+j+k+2)}),\quad
(i,j,k)\in J.
\] Furthermore, the above equations can be greatly simplified by
restricting to the sub-pseudo-group of $\G$ defined by
\begin{equation}
  \label{eq:XxUuUx}  
    X_x = U_u = 1,\quad U_x=X_u = 0. 
\end{equation}
At $p=0$ the transformation law for $Q_{ijk}, k<4$, is, modulo
\eqref{eq:XxUuUx} and $X^{(i+j+1)}, U^{(i+j+1)}$, affine to leading
order
\begin{align*}
  Q_{ij0} &\equiv  q_{ij0} + U_{i+2,j} ,\\
  Q_{ij1} &\equiv q_{ij1} + 2 U_{i+1,j+1} - X_{i+2,j},\\
  Q_{ij2} &\equiv q_{ij2} + 2U_{i,j+2}- 4 X_{i+1,j+1},  \\
  Q_{ij3} &\equiv q_{ij3} - 6 X_{i,j+2} .
\end{align*}
We now impose relations for which $(i,j,k)\in J$.  Setting
$Q_{i,0,0}=Q_{i,j+1,0}=Q_{i+1,j,1}=0$ fixes $U_{i+2,0},U_{i+2,j+1},
X_{i+3,j}$.  Setting $Q_{0j3} = Q_{1j3}=0$ fixes $X_{0,j+2},
X_{1,j+2}$. Fixing $Q_{0,j+1,2} = Q_{1,j+1,2}=0$ determines
$U_{0,j+3},U_{1,j+3}$, and then setting $Q_{0,j+2,1} = 0$ specifies
$X_{2,j+2}$.  Consequently, the remaining relations are
\begin{align*}
  Q_{001}&\equiv q_{001} + 2 U_{11}- X_{20} ,\\
  Q_{011}&\equiv q_{011} + 2 U_{12}- X_{21}, \\
  Q_{102} &\equiv q_{102} + 2U_{12}-4X_{21},\\
  Q_{002} &\equiv q_{002} + 2U_{02}-4X_{11},\\
  Q_{012} &\equiv q_{012} + 2U_{03}-4X_{12}.
\end{align*}
Setting the left-hand side to zero, these relations fix $X_{12},
X_{21}, U_{12} , X_{20}, X_{11}$ leaving $U_{11}$ and $U_{02}$ as free
variables.

We now consider the proof of Proposition \ref{prop:Hstab}.  Let
\[ \E_\Xi= \bs^{-1}(\Xi) \cap \bt^{-1}(\Xi)\subset \tE\] be the
subgroupoid of transformations that preserve $\Xi$, and let $\E_H$ be
the subgroupoid corresponding to the action of $H$ on $\J\ii$.  It is
straightforward to check that $H$ preserves $\Xi$; that is,
$\E_H\subset \E_\Xi$. Equations \eqref{eq:nuxx} also demonstrate that
$\nu_U, \mu_X, \mu_U, \nu_{XU} , \nu_{UU}$ form a basis of one-forms for
the source fibres of $\bs\colon\E_\Xi\to\Xi$.  Since $H$ is
5-dimensional, it follows that $\E_H = \E_\Xi$ by dimensional
exhaustion, as was to be shown.
%%%%%
\end{appendix}
%%%%%

%%%%%

\end{document}